\documentclass[12pt]{article}
\usepackage[a4paper,inner=2.5cm,outer=2.5cm,top=3cm,bottom=2cm]{geometry}
\usepackage{enumitem, tikz}
\usepackage{amsthm, 	mathtools , 
			amssymb, 	stmaryrd  , 
			wasysym, 	bussproofs, 
			faktor , 	amsfonts  }
\usepackage[utf8]{inputenc}
			
\usetikzlibrary{matrix}
\usetikzlibrary{calc}

\newtheorem{theorem}{Theorem}[section]
\newtheorem{prop}[theorem]{Proposition}
\newtheorem{corollary}[theorem]{Corollary}

\newtheorem{fact}[theorem]{Fact}

\newtheorem{definition}[theorem]{Definition}
\newtheorem*{remark}{Remark}

\newcommand{\Open}[1]{\mathcal{O}(#1)}

\newcommand{\topspace}[1]{$\tuple{#1,\Open{#1}}$}

\newcommand{\Set}{\mathbf{Set}}

\newcommand{\Psh}[1]{\mathbf{Psh}\left(#1\right)}

\newcommand{\Sh}[1]{\mathbf{Sh}\left(#1\right)}

\newcommand{\Obj}[1]{\text{Obj}\left(#1\right)}

\newcommand{\Arr}[1]{\text{Arr}\left(#1\right)}

\newcommand{\Seth}[1]{\Set^{\left(#1\right)}}

\newcommand{\chset}{\mathbf{c}\H\text{-}\mathbf{Set}}

\newcommand{\tupleh}[1]{\tuple{#1}^{(\H)}}

\renewcommand{\H}{\mathbb{H}}
\newcommand{\HH}{\mathbb{H}}
\newcommand{\A}{\mathbb{A}}
\newcommand{\B}{\mathbb{B}}

\newcommand{\V}[1]{\text{V}^{(#1)}}

\newcommand{\rar}{\rightarrow}			
\newcommand{\bim}{\leftrightarrow}		

\newcommand{\RAR}{\Rightarrow}			

\newcommand{\sur}{\twoheadrightarrow}	
\newcommand{\inj}{\rightarrowtail}		

\newcommand{\set}[1]{\left\{#1\right\}}
\newcommand{\point}[1]{\left(#1\right)}
\newcommand{\tuple}[1]{\left\langle#1\right\rangle}

\newcommand{\bvalue}[1]{\left\llbracket#1\right\rrbracket}

\newcommand{\hset}{\H\text{-}\mathbf{Set}}    

\makeatletter
\newcommand*{\defeq}{\mathrel{\rlap{%
                     \raisebox{0.3ex}{$\m@th\cdot$}}%
                     \raisebox{-0.3ex}{$\m@th\cdot$}}%
                     =}
\makeatother

\newcommand{\quot}{\faktor}
\newcommand{\dom}[1]{\text{dom}\left({#1}\right)}
\newcommand{\img}[1]{\text{img}\left({#1}\right)}
\newcommand{\fun}[1]{\text{fun}\left({#1}\right)}

\renewcommand{\restriction}{\mathord{\upharpoonright}}

\newcommand{\rank}[1]{\varrho\left(#1\right)}
\newcommand{\rankk}[2]{\varrho'\left(#1, #2\right)}

\newcommand{\ord}{\operatorname{On}}

\newcommand{\LZF}{{\mathcal L_\in}}
\newcommand{\LZFe}{{\LZF^\H}}

\newcommand{\vars}{{\operatorname{Vars}_\LZFe}}
\newcommand{\free}{\operatorname{free}}

\newcommand{\vtovtwo}[1]{{\hat{#1}}}
\newcommand{\vtwotov}[1]{{\check{#1}}}

\usepackage[heads=vee]{diagrams}
\newarrow{Iso}>---{>>}

\begin{document}
\allowdisplaybreaks

\title{
	Induced morphisms between Heyting-valued models
}

\author{ 
	{\large Jos\'e Goudet Alvim}\thanks{jose.alvim@usp.br} \\ 
	{\small Universidade de S\~ao Paulo, S\~ao Paulo, Brasil} \\ 
	{\large  Arthur Francisco Schwerz Cahali}\thanks{arthur.cahali@usp.br}\\ 
	{\small Universidade de S\~ao Paulo, S\~ao Paulo, Brasil}\\
	{\large  Hugo Luiz Mariano}\thanks{hugomar@ime.usp.br}\\ 
	{\small Universidade de S\~ao Paulo, S\~ao Paulo, Brasil} 
}

\date{}
\maketitle
    
	\begin{abstract}
		To the best of our knowledge, there are very few results on how Heyting-valued models are affected 
		by the morphisms on the complete Heyting algebras that determine them: the only cases found in the 
		literature are concerning automorphisms of complete Boolean algebras and complete embeddings between 
		them (\emph{i.e}., injective Boolean algebra homomorphisms that preserve arbitrary suprema and 
		arbitrary infima). In the present  work, we consider and explore how more general kinds of morphisms 
		between complete Heyting algebras $\H$ and $\H'$ induce arrows between $\V\H$ and $\V{\H'}$, and 
		between their corresponding localic topoi $\Seth\H$ ($\simeq \Sh\H$) and $\Seth{\H'}$ ($\simeq\Sh{\H'}$). 
		More specifically: any {\em geometric morphism} $f^* :\Seth\H\to\Seth{\H'}$ (that automatically came from 
		a unique locale morphism  $f : \H \to \H'$), can be ``lifted'' to an arrow $\tilde{f}:\V\H\to\V{\H'}$. 
		We also provide some semantic preservation results concerning this arrow $\tilde{f}:\V\H\to\V{\H'}$.
	\end{abstract}

	\section*{Introduction}

    The expression ``Heyting-valued model of set theory'' has two (related) meanings, both parametrized by 
	a complete Heyting algebra $\H$: 
    
    \textbf{(i)}\phantom{\textbf{i}} \ The canonical Heyting-valued models in set theory, $\V\H$, as introduced in the setting of complete 
	boolean algebras in the 1960s by D. Scott, P. Vop\v{e}nka and R. M. Solovay in an attempt to help understand 
	the then recently introduced notion of \emph{forcing} in ZF set theory developed by P. Cohen 
	(\cite{Jec03}, \cite{Kun11}, \cite{Bel05});

    \textbf{(ii)} \ The (local) ``set-like'' behavior of categories called {\em topoi}, particularly in the case of the 
	(localic) topoi of the form $\Sh\H$ (\cite{Bor08c}, \cite{Bel88}).\\

    The concept of a Heyting/Boolean-valued model is nowadays a general model-theoretic notion, whose 
	definition is independent of forcing in set theory: it is a generalization of the ordinary Tarskian 
	notion of structure where the truth values of formulas are not limited to ``true'' and ``false'', but 
	instead take values in some fixed complete Heyting algebra $\H$. More {precisely}, a $\H$-valued model 
	$M$ in a first-order language $L$ consists of an underlying set $M$ and an assignment $\bvalue{\varphi}_{\H}$ 
	of an element of $\H$ to each formula $\varphi$ with parameters in $M$, satisfying convenient conditions.

    The canonical Heyting-valued model in set theory associated to $\H$ is the pair $\tuple{\V\H, \bvalue{ 
	\ }_{\H}}$, where both components are recursively defined. {Explicitly,} $\V\H$ is the proper class 
	$\V\H := \bigcup_{\beta \in \ord} \V\H_\beta$, where $\V\H_\beta$ is the set of all functions $f$ such 
	that $\dom f \subseteq \V\H_\alpha$, for some $\alpha < \beta$, and $\img{f} \subseteq \H$. Whenever 
	$\H$ is a complete boolean algebra $\tuple{\V\H, \bvalue{ \ }_\H}$ is a model of ZFC in the sense that 
	for each axiom $\sigma$ of ZFC, $\bvalue\sigma_\H = 1_\H$; more generaly, if $\H$ is a complete Heyting 
	algebra, then $\tuple{\V\H, \bvalue{ \ }_\H}$ is a model IZF, the intuitionistic counterpart of ZFC.
    
    On the other hand, $\V\H$ may give rise to a localic topos, $\Seth\H$, that is 
	equivalent to the (Grothendieck) topos $\Sh\H$ of all sheaves over the locale (\mbox{= complete} Heyting algebra) 
	$\H$ (\cite{Bel05}, \cite{Bel88}). The objects of $\Seth\H$ are equivalence classes of members of $\V\H$ 
	and the arrows are (equivalence classes of) members $f$ of $\V\H$ such that ``{$\V\H$} believes, with 
	probability $1_\H$, that $f$ is a function''. A general topos encodes an internal (higher-order) 
	intuitionistic logic, given by the ``forcing-like'' Kripke-Joyal semantics, and some form of (local) 
	set-theory (\cite{Bel88}, \cite{Bor08c}); a localic Grothendieck topos is guided by a more well-behaved 
	internal logic and set theory.

    All the considerations above concern a fixed complete Heyting algebra $\H$. However, to the best of 
	our knowledge, there are very few results on how Heyting-valued models are affected by the morphisms between their algebras. The only cases found in the literature (\cite{Bel05}) 
	are concerning automorphisms of complete Heyting algebras and complete embeddings (\emph{i.e}., injective 
	Heyting algebra homomorphisms that preserves arbitrary suprema and arbitrary infima). In the present work, 
	we consider and explore how more general kinds of morphisms between complete Heyting algebras $\H$ and 
	$\H'$ induce arrows between $\V\H$ and $\V{\H'}$, and between their corresponding Heyting topoi 
	$\Seth{\H}$ ($\simeq \Sh\H$) and $\Seth{(\HH')}$ ($\simeq \Sh{\H'}$). The result is: any \emph{geometric
	morphism} $f^* : \Seth\H \to \Seth{\H'}$ (that automatically came from a unique locale morphism $f:\H\to
	\H'$) can be ``lifted'' to an arrow  $\tilde f: \V\H\to\V{\H'}$.
    
  
	\paragraph{Outline:} 
		In \textbf{Section \ref{sec:prelim}} we provide the main definitions on sheaves over locales (= 
		complete Heyting algebras), topoi and Heyting-valued models of IZF. \textbf{Section \ref{sec:descr}} 
		is devoted to present the equivalent descriptions of the categories of sheaves over a locale, in 
		particular establishing a connection between the cumulative construction of Heyting valued models and 
		localic topoi. \textbf{Section \ref{sec:induced}} contains the main results of this work: the 
		``lifting'' of all geometric morphisms $f^* :\Sh\H\to\Sh{\H'}$ to arrows $\tilde f:\V\H\to\V{\H'}$ 
		and the corresponding semantic preservation results. We end this work in \textbf{Section \ref{sec:final}} 
		presenting some remarks on possible further developments. 
  
  \section{Preliminaries}\label{sec:prelim}
        
        For the reader's convenience, we provide here the main definitions and results on topos and Heyting 
		valued models of set theory. Our main references for category theory are \cite{Lan98} and \cite{Bor08a}; 
		for topos theory \cite{Bor08c} and \cite{Bel88} and for Boolean and Heyting valued models \cite{Bel05}. 
        
        \subsection{Topos and Grothendieck Topos}
         	If \topspace{X} is a topological space, then the family of sets of continuous functions $({\cal C}(U, 
			\mathbb{R}))_{U \in \Open{X}}$ has the property that, for any open subset $U$ and any open covering 
			$U = \bigcup_{i \in I} V_i$ every family of continuous functions $(f_i \in {\cal C}(V_i, \mathbb{R}))_{i \in I}$ 
			that is {\em compatible} (${f_i}_{|V_i \cap V_j} = {f_j}_{|V_i \cap V_j}$, $\forall i,j \in I$), 
			has a unique {\em gluing} $f \in {\cal C}(U, \mathbb{R})$ ($f_{|V_i} = f_i$, $\forall i \in I$): This 
			holds since the property of being continuous is a local property; an analogous remark holds for the 
			${\cal C}^\infty$ functions if $X$ is a smooth manifold. Formally, this is captured by the following:
        
            \begin{definition}
                Let \topspace{X} be a topological space. Regarding the poset $(\Open{X}, \subseteq)$ as a category, 
				a presheaf on $X$ is a functor $F: \Open{X}^{op} \to \Set$. A sheaf on $X$ is a presheaf $F$ such 
				that, for every open $U \in \Open{X}$ and every open covering $\{U_i \in \Open{X} \mid i \in I\}$ 
				of $U$, the diagram below is an equalizer:

				\begin{center}
					\begin{tikzpicture}
						\matrix (m) [matrix of math nodes, row sep=3em, column sep=4em, nodes={text height=1.75ex, text depth=0.25ex}] {
							F(U) & \prod\limits_{i\in I} F(U_i) & \prod\limits_{\tuple{i, j}\in I\times I} F(U_i \cap U_j)\\
						};
						\path[->] (m-1-1) edge (m-1-2);

						\path[->] ($(m-1-2.east) + (0,0.1)$) edge ($(m-1-3.west) + (0,0.1)$);
						\path[->] ($(m-1-2.east) - (0,0.1)$) edge ($(m-1-3.west) - (0,0.1)$);
					\end{tikzpicture}
				\end{center}

                We denote the category of presheaves on $X$ by $\Psh{X}$ and the category of sheaves on $X$ 
				by $\Sh{X}$.
            \end{definition}

            Note that the definition of a sheaf depends only on the lattice of opens, therefore we may define 
			presheaves and sheaves for any {\em locale} $\tuple{\H, \leq}$, \emph{i.e.} a complete lattice 
			satisfying the following  distributive law:
			$$ a \wedge \bigvee_{i\in I} c_i = \bigvee_{i \in I} a\wedge c_i. $$
            Locales are precisely the complete Heyting algebras,  where 
            $$a \rightarrow b = \bigvee\{c \in \HH: a \wedge c \leq b\}.$$
            
            It is also possible to define sheaves in more general categories, using Grothendieck topologies.
            
            \begin{definition}
                Let $\mathcal{C}$ be a small category. A Grothendieck topology on $\mathcal{C}$ is a function 
				$J$ which assigns to each object $c \in \Obj{\mathcal{C}}$ a family $J(c)$ of sieves on $c$, 
				satisfying:
				
                \begin{enumerate}
                    \item 	Maximal sieve: $\bigcup\limits_{a \in \mathcal{C}_0} \mathcal{C}(a,c) \in J(c)$, 
							for all $c \in \Obj{\mathcal{C}}$;

                    \item 	Pullback stability: given $c \in \Obj{\mathcal{C}}$, for every $S \in J(c)$ and 
							every arrow $f: a \to c$, the pullback $f^\ast S$ of the sieve $S$ along $f$ is 
							an element of $J(a)$;

                    \item 	Transitivity: given $c \in \Obj{\mathcal{C}}$ and $S \in J(c)$, if $R$ is a sieve 
							on $c$ such that $f^\ast R$ is a sieve on $a$ for every $f: a \to c$ in $S$, then 
							$R \in J(c)$.
                \end{enumerate}

                We call the pair $(\mathcal{C},J)$ a (small) site.
            \end{definition}
            
            Every locale $(\HH, \leq)$ gives rise to a Grothendieck topology: if $c \in \HH$, then $J(c)$ is 
			the set of all coverings of $c$ that are downward closed. Another important example is the Zariski 
			topology in algebraic geometry.
            
            \begin{definition}
                Let $(\mathcal{C},J)$ be a site and $F: \mathcal{C}^{op} \to \Set$ a presheaf. We say $F$ is a 
				sheaf on $(\mathcal{C},J)$ if, for all $c \in \Obj{\mathcal{C}}$, every collection $\{f_i: a_i 
				\to c \mid i \in I\} \in J(c)$ and every $F$-compatible family $\{s_i\in F(a_i)\mid i\in I\}$, 
				there exists a unique $s \in F(c)$ such that $F(f_i)(s) = s_i$, for all $i \in I$. We denote 
				the category of sheaves on $(\mathcal{C},J)$ by $\Sh{\mathcal{C},J}$.
            \end{definition}
            
        
            A Grothendieck topos is a category that is equivalent to the topos of sheaves on a site. Apart from categories of sheaves, the category of presheaves is also an example of a Grothendieck topos.
            
            Some properties of Grothendieck topos are of particular interest for developing logic in the context of category theory, such as containing a subobject classifier and being Cartesian closed.
            
            \begin{definition}
                Given a category $\mathcal{C}$ with pullbacks, a subobject classifier in $\mathcal{C}$ is a mono $\top: u \inj \Omega$ satisfying: for every other mono $m: a \inj b$, there exists a unique $\chi_m: b \to \Omega$ such that the following diagram is a pullback:
                \begin{diagram}
                    a & \rEmbed^m & b\\
                    \dTo & & \dTo>{\chi_m}\\
                    u & \rEmbed^\top & \Omega
                \end{diagram}
            \end{definition}
            
            \begin{definition}
                Let $\mathcal{C}$ be a locally small category with binary products. The category $\mathcal{C}$ is called Cartesian closed if, for every $b \in \Obj{\mathcal{C}}$, the product functor $- \times b$ has a right adjoint (the exponentiation functor $(-)^b$).
            \end{definition}
            
            Both these properties are used to define a more general notion of topos: elementary topos.

        \subsection{Localic Topos}
        
            \begin{definition}
            A topos is said to be localic if it is equivalent to the topos of sheaves on a locale.
            \end{definition}
            
            \begin{theorem}
            For a Grothendieck topos $\mathcal{T}$, the following conditions are equivalent:
            \begin{enumerate}
                \item $\mathcal{T}$ is a localic topos;
                \item the subobjects of the terminal object constitute a family of generators of $\mathcal{T}$.
            \end{enumerate}
            \end{theorem}
            
            \begin{theorem}
            For a Grothendieck topos $\mathcal{T}$, the following conditions are equivalent:
            \begin{enumerate}
                \item $\mathcal{T}$ is a localic and Boolean topos;
                \item $\mathcal{T}$ satisfies the axiom of choice;
                \item $\mathcal{T} \simeq \Sh{\B}$, for some Boolean algebra $\B$.
            \end{enumerate}
            \end{theorem}
            
            A continuous function between topological spaces defines a $\point{\wedge, \bigvee}$-preserving morphism (\emph{i.e.}, the left adjoint of a morphism of locales) between the locales of open sets, and a geometric morphism between the corresponding sheaf topos:
            \begin{center}
                \begin{tabular}{c c c c c}
                    \begin{diagram}
                        X \\
                        \dTo>f \\
                        X'
                    \end{diagram}& \(\qquad \mapsto \qquad\) & 
                    \begin{diagram}
                        \Open{X}\\
                        \uTo>{f^{-1}}\\
                        \Open{X'}
                    \end{diagram}& \(\qquad \mapsto \qquad\) & 
                    \begin{diagram}
                        \Sh{\Open{X}}\\
                        \dTo<{\varphi_\ast} \uTo>{\varphi^\ast}\\
                        \Sh{\Open{X'}}
                    \end{diagram}
                \end{tabular}
            \end{center}
            That is, $(\varphi_\ast,\varphi^\ast)$ is a pair of functors such that $\varphi^\ast \dashv \varphi_\ast$ and $\varphi^\ast$ preserves finite limits.
            
            This mapping from the category of topological spaces to the category of topos and geometric morphisms is not full nor faithful. However, the mapping from the category of locales to the category of topos and geometric morphisms is fully faithful:
            
            \begin{theorem}
            The mapping below, given by $\varphi_\ast(F) = F \circ f^\ast$, for every sheaf $F$ on $\H$,
            \begin{center}
                \begin{tabular}{c c c}
                    \begin{diagram}
                        \H\\
                        \dTo<{f_\ast} \uTo>{f^\ast \quad (\wedge,\bigvee)}\\
                        \H'
                    \end{diagram} & \(\qquad \mapsto \qquad \) &
                    \begin{diagram}
                        \Sh{\H}\\
                        \dTo<{\varphi_\ast} \uTo>{\varphi^\ast}\\
                        \Sh{\H'}
                    \end{diagram}
                \end{tabular}
            \end{center}
            defines a fully faithful functor $\mathbf{Sh}: \mathbf{Loc} \to \mathbf{Topos}_{geo}$.
            \end{theorem}

        
        \subsection{Locale-Valued Models}
            
            
            \begin{definition}{Locale-Valued Model}
            
                We define, for a locale $\H$, the universe of $\H$-names by ordinal recursion. Given an ordinal $\alpha$ let
                
                $$ \V{\H}_\alpha = \set{f\in \H^X \mid \exists\beta<\alpha, X\subseteq\V{\H}_\beta 
                } $$
                
                It is readily seen that $\V{\H}_{\alpha}\subset\V{\H}_{\alpha+1}$ and that for limit ordinals it is simply the union of the earlier stages. So we let the (proper class) $\V{\H}$ be defined as:
                
                $$ \V\H = \bigcup_{\alpha \in \ord} \V\H_\alpha $$
                
                Furthermore, we define for elements of this universe $\V\H$ a function $\rank x$ defined as:
                
                $$ \rank x = \min \set{\alpha \in\ord \mid x\in\V\H_\alpha } $$
                
                Which is trivially well-founded.
            \end{definition}
            
            \begin{definition}{Atomic Formulas' Values}
                
                We endow this class with two\footnote{Technically three.} binary functions on $\H$, namely $\bvalue{\cdot\in\cdot}$ and $\bvalue{\cdot = \cdot}$ defined by simultaneous recursion on a well-founded relation.
                
                Given a locale $\H$
                define first:
                
                $$\tuple{x,y}\prec\tuple{u,v}\iff (x=u\land y\in\dom v) \lor (x\in\dom u \land y=v)$$
                
                We will later see that this is a well-founded relation on $\V\H\times\V\H$. By recursion on $\prec$, define for pairs of elements in $\V\H$, the following:
                
                \begin{align*}
                    \bvalue{\cdot\in\cdot} : \V\H&\times\V\H\longrightarrow \H\\
                    \langle x&, y\rangle\longmapsto \bigvee_{u\in\dom y} y(u)\wedge \bvalue{x=u}
                \end{align*}
                
                \begin{align*}
                    \bvalue{\cdot\ni\cdot} : \V\H&\times\V\H\longrightarrow \H\\
                    \langle x&, y\rangle\longmapsto \bigvee_{v\in\dom x} x(v)\wedge \bvalue{v=y}
                \end{align*}
                
                \begin{align*}
                    \bvalue{\cdot = \cdot} : \V\H&\times\V\H\longrightarrow \H\\
                    \langle x&, y\rangle\longmapsto \bigwedge_{\substack{u\in\dom y\\ v\in\dom x}} \left(y(u)\rar \bvalue{x\ni u}\right)\wedge\left(x(v)\rar\bvalue{v\in y}\right)
                \end{align*}
                
                We call these the $\H$-values of the membership, co-membership and, equality, respectively.
                Where ambiguity may arise, we make a distinction between valuations of different locale-valued models by subscripting the 
                relevant locale on $\bvalue{\ }$ or by simply placing a prime symbol on one of the otherwise ambiguous evaluation function: “$\bvalue{\ }'$”.
            \end{definition}
            
            \begin{prop} {The relation $\prec$ is  well-founded.} 
            \end{prop}
                
                \begin{proof}
                
                Firstly, define $\rankk{x}{y}=\min\set{\rank{x},\rank{y}}$. Take any subclass $X$ of $\V\H\times\V\H$ and consider its image under $\varrho'$. For one, it has a minimum, due to the well-orderedness of $\ord$, let us call this minimum value $\alpha$ and one pair in $X$ such that it attains value $\alpha$ we name $\tuple{x,y}$. Without loss of generality, we assume $\rank x = \alpha$.
                
                Suppose now there is some $\tuple{f,g}$ such that $\tuple{f,g}\prec\tuple{x,y}$. By definition:
                
    			$$ (f\in\dom x\land g=y)\lor(f=x\land g\in\dom y) $$
    
                Breaking the disjunction we realize that the value of $\tuple{f,g}$ under $\varrho'$ must be no more than $\alpha$:
    
    			\begin{prooftree}
    				\AxiomC{$f\in\dom x\land g=y$}
    				\UnaryInfC{$\rankk{f}{g}\leq\rankk{x}{y}$}
    
    				\AxiomC{$f=x\land g\in\dom y$}
    				\UnaryInfC{$\rankk{f}{g}\leq\rankk{x}{y}$}
    
    				\BinaryInfC{$\rankk{f}{g}\leq\rankk{x}{y}$}
    			\end{prooftree}
    
                But since $\alpha$ is minimal, it follows that they must be the same. Therefore, $f$ can't be in the domain of $x$, for then it would have a smaller rank than $\alpha$. One is forced to conclude that, under this assumption, $\tuple{f,g}\prec\tuple{x,y}\rar g\in\dom y$, and since it is necessary that $f=x$. Consequently, because the relation $(\cdot\in\dom\cdot)$ is well-founded, one is forced to concede that if a descending $\prec$-chain does not stabilize, so too will a descending $(\cdot\in\dom\cdot)$-chain, a contradiction.
                        \end{proof}
            
            We are thus entitled to make the definition of those $\H$-values as we previously claimed. The definitions of $\H$-values are then extended to the class of the language of $\LZF$-formulas enriched with constant symbols for each member of the class $\V\H$:
            
            \begin{definition}{$\H$-valuation of Formulas}
            
                We define the value of a $\LZFe$-formula $\varphi$ which is the language $\LZF$ extended by constant symbols for each element of $\V\H$ the valuation $\bvalue{\varphi}$ inductively on the complexity of $\varphi$. 
                
                $$ \bvalue{\varphi}:(\vars\to\V\H)\to\H $$
                
                For an atomic formula $\varphi$ involving only constants symbols as terms, $\bvalue{\varphi}$ is simply the value of the corresponding function defined earlier, that is if $\varphi\equiv aRb$ then, its valuation is the constant function:
                
                $$ \bvalue{\varphi} = \bvalue{aRb} $$
                
                For atomic formulas with free variables, the valuation is a function of the language's variable symbols. Given a function that assigns values to its free variables, it yields the value of the corresponding valuation, \emph{i.e.} if $\varphi\equiv x_i R x_j$ for variable symbols $x_i$, $x_j$.
                
                \begin{align*}
                    \bvalue{\varphi} : \left(\vars\to\V\H\right)&\to\H\\
                    v&\mapsto\bvalue{v(x_i) R v(x_j)}
                \end{align*}
                
                If $\varphi$ has mixed constants and variable symbols, the definition is analogous but fixing the constants.
                
                When $\varphi$ has free variables, we often write $\bvalue{\varphi(x,y\cdots)}$ to denote the function on its free variables, rather than writing $\bvalue{\varphi}(\cdots)$.
                
                For complex formulas, we define, for negation:
                
                $$ \bvalue{\neg\varphi} = \bvalue{\varphi}\rar\bot$$
                
                For a binary connective $*$ among $\rar$, $\wedge$, $\vee$:
                
                $$ \bvalue{\varphi * \psi} = \bvalue{\varphi} * \bvalue{\psi}$$
                
                Define now, given $x_i$ a variable symbol and $x\in\V\H$ and a function $v:\vars\to\V\H$, the function $$v_{[x_i\mid x]}(s) =
                    \begin{cases}
                        v(s)\text{, if }s\not=x_i\\
                        x\text{, otherwise}
                    \end{cases}$$
                
                For quantifiers:
                
                \begin{align*}
                    \bvalue{\forall x_i: \varphi} : \left(\vars\to \V\H\right)&\to\H\\
                    v&\mapsto\bigwedge_{x\in\V\H}\bvalue{\varphi}\left(v_{[x_i \mid x]}\right)
                \end{align*}
                
                And dualy,
                
                \begin{align*}
                    \bvalue{\exists x_i: \varphi} : \left(\vars\to \V\H\right)&\to\H\\
                    v&\mapsto\bigvee_{x\in\V\H}\bvalue{\varphi}\left(v_{[x_i \mid x]}\right)
                \end{align*}
                
                If two functions $f,g:\vars\to\V\H$ coincide on $\free_\varphi$, then $\bvalue{\varphi}(f) = \bvalue{\varphi}(g)$, then the sentences are constant functions --- and we will often omit the fact that these valuations are indeed functions of the values we assign to the variable symbols of the formulas and concern ourselves with sentences, which correspond to values in $\H$. 
                
                We won't distinguish between the constant symbols corresponding to elements of $\V\H$ and \emph{the} elements they correspond to.
            \end{definition}
            
            We simply state the following results without proof --- as they can be straightforwardly adapted from those on \cite{Bel05}.
            
            
            \begin{theorem}{Properties of Formula valuation}
    	        \begin{enumerate}
    		        \item $\bvalue{x  =  x} = 1$.																	
                    \item $\forall x\in \dom y: y(x) 			   		\leq \bvalue{x \in y}$. 					
                    \item $\bvalue{x  =  y} = \bvalue{y = x}$.														
        			\item $\bvalue{x \in y} = \bvalue{y \ni x}$.													
                    \item $\bvalue{x  =  y} \wedge \bvalue{y = z} 		\leq \bvalue{x  =  z}$.						
                    \item $\bvalue{x  =  y} \wedge \bvalue{y \in z} 		\leq \bvalue{x \in z}$. 					
                    \item $\bvalue{x \in y} \wedge \bvalue{y = z}		\leq \bvalue{x \in z}$. 					
                    \item $\bvalue{x  =  y} \wedge  x(u)      	   		\leq \bvalue{u \in y}$. 					
        			\item $\bvalue{x  =  y} \wedge \bvalue{\varphi(x)}     =  \bvalue{y  =  x}\wedge \bvalue{\varphi(y)}$.
        			\item $\bvalue{\exists u\in x:\varphi(u)} = \bigvee_{  u\in\dom x}x(u)\wedge\bvalue{\varphi(u)}$	
        			\item $\bvalue{\forall u\in x:\varphi(u)} = \bigwedge_{u\in\dom x}x(u)\rar \bvalue{\varphi(u)}$	
        		\end{enumerate}
        		
        		Where $Q u\in x: \varphi(u)$ is the usual shorthand for either $[\forall u:u\in x\rar \varphi(u)]$ or $[\exists u:u\in x\land \varphi(u)]$ for $Q$ standing for “\,$\forall$” or “\,$\exists$”.
            \end{theorem}
            
            We then must define a ``localic semantic'', or a notion of truth for that structure so that we may claim that it actually models some form of set theory.
            
            \begin{definition}{$\H$-Semantic / $\H$-Validity / Localic Semantic}
            
                We define a Tarskian-like $\vDash$ for each $\H$ to say:
                
                $$\V\H\vDash\varphi \iff \bvalue{\varphi}=\top $$
                
                Since $\bvalue{\varphi}$ is a function in disguise, to properly make this comparison, it either must be constant or we must convert $\top$ to the constantly tautological function. The latter allows us to interpret formulas with free variables, and will assign them truth if they are always true under any valuation.
                
                Hence, we extend the notion for, given $\sigma:\vars\to\V\H$, 
                
                $$ \V\H\vDash_\sigma \varphi \iff \bvalue{\varphi}(\sigma) = \top$$
            \end{definition}
            
            \begin{prop}{Some properties of the localic $\vDash$}
                        
        		\begin{enumerate}
        			\item $\V\H\vDash_\sigma\varphi$ and $\V\H     \vDash_\sigma\psi$ $\iff$ $\V\H    \vDash_\sigma\varphi\land\psi$
        			\item $\V\H\vDash_\sigma\forall x:\varphi(x)$                 	$\iff$ for all  $X$, $\V\H\vDash_{\sigma[x\mid X]}\varphi(x)$
        			\item $\V\H\vDash_\sigma\exists x:\varphi(x)$                 	$\iff$ for some $X$, $\V\H\vDash_{\sigma[x\mid X]}\varphi(x)$
        			\item $\V\H\vDash_\sigma\neg\varphi$                          	$\RAR$ $\V\H\not\vDash_\sigma\varphi$
        			\item $\V\H\vDash_\sigma\varphi$ or $\V\H    \vDash_\sigma\psi$ $\RAR$ $\V\H    \vDash_\sigma\varphi\lor\psi$
        			\item $\V\H\vDash_\sigma\psi$ or $\V\H\vDash_\sigma\neg\varphi$ $\RAR$ $\V\H    \vDash_\sigma\varphi\rar\psi$
        		\end{enumerate}
        		
        		Furthermore, modus ponens; generalization; instances of intuitionistic tautology (or classical tautologies if $\H$ is Boolean) and the intuitionistic first order logic axioms are all valid under the eyes of our $\H$-validity.
        		
        		In fact, $\H$-$\vDash$ is sound with respect to $\vdash$.
            \end{prop}
            
            
            
            \begin{remark}
                The difference between the Tarskian and this Localic semantic is that some equivalences that hold for Tarski's do not hold in Localic semantic (in the nontrivial cases, \emph{i.e.}, $\H\neq\set{0}$ or $\H\neq2$). Also, there is very little reason -- a priori -- to expect there to be a witness to an existential formula, as it is the arbitrary supremum of values of other formulas.
                
                The supremum in existential formulas may be attained by witnesses, but this relies on the additional property of $\H$ being Boolean, otherwise only a weaker statement holds. 
            \end{remark}
            
            There is, in fact, a canonical representation of elements of our universe $V$ and the many universes $\V\H$. For the case $\H=2=\set{0,1}$, this canonical representation establishes a sort of model equivalence, has a good left inverse and is onto modulus $\V\H$-equality with value $\top$.
            
            \begin{definition} \label{immer-def}{Immersion of $V$ in    $\V\H$}
                
                Define, $\in$-recursively:
                
                $$\vtovtwo{x} = \set{ \tuple{\vtovtwo y, \top} \mid y\in x }$$
                
                Also, for $x\in\V2$ define $(\cdot\in\dom\cdot)$-recursively:
                
                $$\vtwotov{x} = \set{ \vtwotov y \mid y\in\dom x \land x(y)= \top  } $$
            \end{definition}
            
            
            \begin{prop} \label{immer-prop} The following hold:
            
                \begin{enumerate}
        			\item $\forall x, y		 	: x \in y  \bim \bvalue{\vtovtwo x\in \vtovtwo y} = 1$
        			\item $\forall x', y'\in\V2 : \vtwotov{x'}\in\vtwotov{y'} \bim \bvalue{ x'\in y'} = 1$
        
        			\item $\forall x, y		 	: x   = y  \bim \bvalue{\vtovtwo x =  \vtovtwo y} = 1$
        			\item $\forall x', y'\in\V2 : \vtwotov{x'}  =\vtwotov{y'} \bim \bvalue{ x' =  y'} = 1$
        
        			\item $\forall x 			: \vtwotov{\vtovtwo x} = x$
        			\item $\forall x' \in\V2 	: \bvalue{\vtovtwo{\vtwotov {x'}} = x'} = 1$
    		    \end{enumerate}
    
            \end{prop}
            
            \begin{corollary} Suppose  $\H \neq \{0\}.$ \label{immer-cor}
            \begin{enumerate}
             
                \item 
                $V\vDash_\sigma \varphi\iff\V2\vDash_{\vtovtwo{\_}\circ\sigma}\varphi$
                
                   \item  $\varphi\in\Sigma_0\RAR\left[ V\vDash_\sigma \varphi \iff \V\H\vDash_{\vtovtwo{\_}\circ\sigma}\varphi\right]$
                   
                   \item  $\varphi\in\Sigma_1\RAR\left[ V\vDash_\sigma \varphi \RAR \V\H\vDash_{\vtovtwo{\_}\circ\sigma}\varphi\right]$
            \end{enumerate}
            \end{corollary}

            \begin{theorem}{$\V\H$ are models of Intuitionistic Set Theory.} Furthermore, if $\H$ is Boolean, it validates classical set theory and the Axiom of Choice (provided the base universe already did).
            
                This is to say, for all $\varphi$ axioms of the appropriate theory: 
            
                $$ \V\H \vDash \varphi$$
            
                Again we provice no proof since this result is well established (\cite{Bel05}). 
            \end{theorem}

    
    
    \section{$\V{\H}$ and Equivalent Descriptions of $\Sh{\H}$}\label{sec:descr}
    
    In this section, that is based on \cite{Bel05} and \cite{Bor08c}, we present some equivalent descriptions of the category of sheaves of a complete Heyting algebra $\HH$, $\Sh{\HH} \simeq \hset \simeq \Seth{\H}$. This is not only for the reader's convenience, but also because we will later need a detailed description of the equivalence $\hset \simeq \Seth{\H}$, which is only sketched in the appendix of \cite{Bel05}. We start by providing the definitions of these categories.
    
    
        \begin{definition}
            Consider the equivalence relation in $\V{\H}$ given by $f \equiv g$ if, and only if, $ \bvalue{f=g} = 1$. The category $\Seth{\H}$ is defined as:
            $$\Obj{\Seth{\H}} \defeq \quot{\V{\H}}{\equiv}$$
            $$\Seth{\H}\left([x],[y]\right) \defeq \set{[\phi] \in \Seth{\H} \ \middle| \ \bvalue{\fun{\phi: x \to y}} = 1}$$
            The arrows do not depend on the choice of representative of the equivalence classes $[x]$ and $[y]$. The composition and identity are defined as in $\Seth{\H}$ but with the quotient being taken.
        \end{definition}
        
        \begin{definition}
            A $\H$-set is a pair $\tuple{X,\delta}$ such that $X$ is a set and $\delta: X \times X \to \H$ satisfies, for every $x,y,z \in X$,
            \begin{enumerate}
                \item $\delta(x,y) = \delta(y,x)$;
                \item $\delta(x,y) \wedge \delta(y,z) \leq \delta(x,z)$.
            \end{enumerate}
        \end{definition}
        
        \begin{definition}
            A morphism $\phi: \tuple{X,\delta} \to \tuple{X',\delta'}$ of $\H$-sets is a function $\phi: X \times X' \to \H$ such that, for all $x,y \in X$ e $x',y' \in X'$:
            \begin{enumerate}
                \item $\delta'(x',y') \wedge \phi(x,y') \leq \phi(x,x')$;
                \item $\delta(x,y) \wedge \phi(x,y') \leq \phi(y,y')$;
                \item $\phi(x,x') \wedge \phi(x,y') \leq \delta'(x',y')$;
                \item $\bigvee\limits_{z' \in X'} \phi(x,z') = \delta(x,x)$.
            \end{enumerate}
            A morphism of $\H$-sets, then, can be understood as an $\H$-valued functional relation.
        \end{definition}
        
        Given morphisms $\phi: \tuple{X,\delta} \to \tuple{X',\delta'}$ and $\psi: \tuple{X',\delta'} \to \tuple{X'',\delta''}$ of $\H$-sets, their composition $\psi \circ \phi$ is given by:
        $$(\psi \circ \phi)(x,x'') = \bigvee\limits_{x' \in X'} \phi(x,x') \wedge \psi(x',x'')$$
        for all $x \in X, x'' \in X''$. The identity morphism $id_{\tuple{X,\delta}}$ is the function such that:
        $$id_{\tuple{X,\delta}}(x,y) = \delta(x,y), \text{ for all } x,y \in X$$
        Thus, we can define the category $\hset$, of $\H$-sets and their morphisms.
        
        One result on morphisms of $\H$-sets in particular will be useful later on:
        
        \begin{prop}\label{prop:igualdade-morf-hsets}
            Given morphisms $\phi,\psi: \tuple{X,\delta} \to \tuple{X',\delta'}$ of $\H$-sets, the following conditions are equivalent:
            \begin{enumerate}
                \item $\phi = \psi$;
                \item $\phi(x,x') \leq \psi(x,x')$, for all $x \in X$ and $x' \in X'$.
            \end{enumerate}
        \end{prop}
        
        \begin{definition}
            A singleton of an $\H$-set $\tuple{X,\delta}$ is a mapping $\sigma: X \to H$ such that, for every $x,y \in X$,
            \begin{enumerate}
                \item $\sigma(x) \wedge \sigma(y) \leq \delta(x,y)$;
                \item $\sigma(x) \wedge \delta(x,y) \leq \sigma(y)$.
            \end{enumerate}
            Note that, given $x \in X$, the function $\sigma_x: X \to H$ such that $\sigma_x(y) = \delta(x,y)$, for all $y \in H$, defines a singleton.
        \end{definition}
        
        \begin{definition}
            Consider $\sigma(X)$ the collection of singletons of an $\H$-set $\tuple{X,\delta}$. $\tuple{X,\delta}$ is said to be complete if the function $\Upsilon: X \to \sigma(X)$, given by $\Upsilon(x) = \sigma_x$, for all $x \in X$, is bijective. We denote the full subcategory of complete $\H$-sets by $\chset$.
        \end{definition}
        
        There is also an alternative description of complete $\H$-sets:
        
        \begin{prop}
            $\chset$ is isomorphic to the category whose objects are complete $\H$-sets and arrows are functions $f: X \to X'$ such that:
            \begin{enumerate}[topsep=0pt]
                \item $\delta(x,y) \leq \delta'(f(x),f(y))$;
                \item $\delta(x,x) = \delta'(f(x),f(x))$;
            \end{enumerate}
            for all $x,y \in X$. The composition is given by usual function composition, and the identity arrow is the identity function.
        \end{prop}
        
        \begin{theorem}
            Let $\tuple{X,\delta}$ be an $\H$-set. Define the $\H$-set $\tuple{\sigma(X), \sigma(\delta)}$ where
            $$\sigma(\delta)(\rho,\tau) = \bigvee\limits_{x \in X} \rho(x) \wedge \tau(x), \text{ for all } (\rho,\tau) \in \sigma(X) \times \sigma(X)$$
            Then, $\tuple{\sigma(X),\sigma(\delta)}$ is complete and isomorphic to $\tuple{X,\delta}$.
        \end{theorem}
        
        The inverse isomorphisms $\phi: \tuple{X,\delta} \to \tuple{\sigma(X), \sigma(\delta)}$ e $\psi: \tuple{(\sigma(X), \sigma(\delta)} \to \tuple{X,\delta}$ are given by:
        $$\phi(x,\rho) = \rho(x), \text{ for all } (x,\rho) \in X \times \sigma(X)$$
        $$\psi(\rho,x) = \rho(x), \text{ for all } (\rho,x) \in \sigma(X) \times X$$
        
        \begin{corollary}
            There is an equivalence of categories: $\hset \simeq \chset$.
        \end{corollary}
        
        We can thereby define the functor $\Gamma: \Sh{\H} \to \chset$ by:
        \begin{multline*}
            \Gamma(F) = \tuple{X_F, \delta_F}, \text{ for every sheaf } F \text{ on } \H, \text{ where } X_F \defeq \coprod\limits_{a \in \H} F(a) \text{ and}\\
            \delta_F \text{ is given by } \tuple{(s,b),(t,c)} \mapsto \bigvee \set{ d \leq b \wedge c \mid s\restriction^b_d = t\restriction^c_d }
        \end{multline*}
        \begin{multline*}
            \Gamma(\eta): \Gamma(F) \to \Gamma(G), \text{ for every natural transformation } \eta: F \RAR G\\
            \text{in } \Sh{\H}, \text{where } \Gamma(\eta)(s,b) = (\eta_b(s),b), \text{ for all } (s,b) \in \Gamma(F)
        \end{multline*}
        
        \begin{theorem}
            The functor $\Gamma: \Sh{\H} \to \chset$ defined above is fully faithful, and for all complete $\H$-set $\tuple{X, \delta}$ there exists a sheaf $F$ on $\H$ such that $\tuple{X, \delta} \cong \Gamma(F)$. Therefore, $\Gamma$ defines an equivalence of categories $\Sh{\H} \simeq \chset$.
        \end{theorem}
        
        Finally, to show the equivalence between $\hset$ and $\Seth{\H}$\footnote{Here we follow \cite{Bel05}, but we provide a more complete and accurate description.}, we will need two constructions on $\V{\H}$:
        
        
        Let $\tuple{X, \delta}$ be an $\H$-set. For each $x \in X$, define $\dot{x} \in \V{\H}$ as:
        $$\dom{\dot{x}} \defeq \set{\hat{z} \mid z \in X} \quad \text{ and } \quad \dot{x}(\hat{z}) \defeq \delta(x,z), \text{ for all } z \in X$$
        Then, define $X^\dagger \in \V{\H}$ as
        $$\dom{X^\dagger} \defeq \set{\dot{x} \mid x \in X} \quad \text{ and } \quad X^\dagger(\dot{x}) \defeq \delta(x,x), \text{ for all } x \in X$$
        Similarly, given a morphism $\phi: \tuple{X,\delta} \to \tuple{X',\delta'}$ of $\H$-sets, we may consider $\varphi^\dagger \in \V{\H}$ given by:
        $$\dom{\phi^\dagger} \defeq \set{ \tupleh{ \dot{x},\dot{x}' } \ \middle| \ x \in X, x' \in X' }$$
        $$\phi^\dagger\left( \tupleh{ \dot{x},\dot{x}' } \right) \defeq \phi(x,x'), \text{ for all } x \in X, x' \in X'$$
        Since $\V{\H} \models \fun{\phi^\dagger}$, we may define a functor $\Phi: \hset \to \Seth{\H}$ by taking $\Phi(X,\delta) = \left[X^\dagger\right]$, for every $\H$-set $\tuple{X,\delta}$, and $\Phi(\phi) = \phi^\dagger$, for every arrow $\phi \in \Arr{\hset}$.
        
        On the other hand, given $u \in \V{\H}$, define $X_u \defeq \dom{u}$ and $\delta_u: X_u \times X_u \to \H$ as
        $$\delta_u(x,y) \defeq \bvalue{x \in u} \wedge \bvalue{x = y } \wedge \bvalue{y \in u}, \text{ for all } x,y \in X_u$$
        Observe, however, that $\bvalue{ u = u' } = 1$ does not imply $X_u = \dom{u} = \dom{u'} = X_{u'}$, and that we may not define a $\H$-set using $[\dom{u}]$ since this class is not a set (later we will show that $\set{u' \in \V{H} \mid \bvalue{ u = u' } = 1}$ is a proper class). In that case, we will use Scott's trick to define a functor $\Psi: \Seth{\H} \to \hset$.
        
        Firstly, if $\bvalue{ u = u' } = 1$, then $\tuple{X_u, \delta_u} \cong \tuple{X_{u'}, \delta_{u'}}$. Indeed, define $\lambda_{u,u'}: \tuple{X_u, \delta_u} \to \tuple{X_{u'}, \delta_{u'}}$ such that
        $$\lambda_{u,u'}(x,x') \defeq \bvalue{ x \in u } \wedge \bvalue{ x = x' } \wedge \bvalue{ x' \in u' }, \text{ for all } x \in \dom{u}, x' \in \dom{u'}$$
        We verify this is a morphism of $\H$-sets. Let $x,y \in X_u$ and $x',y' \in X_{u'}$.
        
        \begin{enumerate}
            \item $\delta_{u'}(x',y') \wedge \lambda_{u,u'}(x,y') \leq \lambda_{u,u'}(x,x')$. Indeed,
            \begin{align*}
                \delta_{u'}(x',y') &\wedge \lambda_{u,u'}(x,y') =\\
                 &= \bvalue{ x' \in u' } \wedge \bvalue{ x' = y' } \wedge \bvalue{ y' \in u' } \wedge \bvalue{ x \in u } \wedge \bvalue{ x = y' } \wedge \bvalue{ y' \in u' } \\
                 &\leq \bvalue{ x' \in u' } \wedge \bvalue{ x' = y' } \wedge \bvalue{ x \in u } \wedge \bvalue{ x = y' } \\
                 &\leq \bvalue{ x \in u } \wedge \bvalue{ x = x' } \wedge \bvalue{ x' \in u' } \\
                 &= \lambda_{u,u'}(x,x')
            \end{align*}
            \item $\delta_u(x,y) \wedge \lambda_{u,u'}(x,y') \leq \lambda_{u,u'}(y,y')$. Indeed,
            \begin{align*}
                \delta_u(x,y) &\wedge \lambda_{u,u'}(x,y') = \\
                 &= \bvalue{ x \in u } \wedge \bvalue{ x = y } \wedge \bvalue{ y \in u } \wedge \bvalue{ x \in u } \wedge \bvalue{ x = y' } \wedge \bvalue{ y' \in u' } \\
                 &\leq \bvalue{ x = y } \wedge \bvalue{ y \in u } \wedge \bvalue{ x = y' } \wedge \bvalue{ y' \in u' } \\
                 &\leq \bvalue{ y \in u } \wedge \bvalue{ y = y' } \wedge \bvalue{ y' \in u' } \\
                 &= \lambda_{u,u'}(y,y')
            \end{align*}
            \item $\lambda_{u,u'}(x,x') \wedge \lambda_{u,u'}(x,y') \leq \delta_{u'}(x',y')$. Indeed,
            \begin{align*}
                \lambda_{u,u'}(x,x') &\wedge \lambda_{u,u'}(x,y') = \\
                 &= \bvalue{ x \in u } \wedge \bvalue{ x = x' } \wedge \bvalue{ x' \in u' } \wedge \bvalue{ x \in u } \wedge \bvalue{ x = y' } \wedge \bvalue{ y' \in u' } \\
                 &\leq \bvalue{ x = x' } \wedge \bvalue{ x' \in u' } \wedge \bvalue{ x = y' } \wedge \bvalue{ y' \in u' } \\
                 &\leq \bvalue{ x' \in u' } \wedge \bvalue{ x' = y' } \wedge \bvalue{ y' \in u' } \\
                 &= \delta_{u'}(x',y')
            \end{align*}
            \item $\bigvee\limits_{z' \in X_{u'}} \lambda_{u,u'}(x,z') = \delta_u(x,x)$. Indeed, using that $\delta_u(x,x) = \bvalue{ x \in u }$,
            \begin{itemize}
                \item on one hand, for every $z' \in X_{u'}$,
                $$\lambda_{u,u'}(x,z') = \bvalue{ x \in u } \wedge \bvalue{ x = z' } \wedge \bvalue{ z' \in u' } \leq \bvalue{ x \in u }$$
                Therefore, $\bigvee\limits_{z' \in X_{u'}} \lambda_{u,u'}(x,z') \leq \delta_u(x,x)$;
                \item on the other hand, for every $z \in X_{u'}$,
                \begin{multline*}
                    u'(z') \wedge \bvalue{ z' = x } = \bvalue{ x = z'} \wedge u'(z') \wedge \bvalue{ z' = z' } \leq \\
                    \leq \bvalue{ x = z'} \wedge \left( \bigvee\limits_{t' \in X_{u'}} u'(t') \wedge \bvalue{ t' = z' } \right) = \bvalue{ x = z'} \wedge \bvalue{ z' \in u' }
                \end{multline*}
                Thus,
                $$\bvalue{ x \in u' } = \bigvee\limits_{z' \in X_{u'}} u'(z') \wedge \bvalue{ z' = x } \leq \bigvee\limits_{z' \in X_{u'}} \bvalue{ x = z'} \wedge \bvalue{ z' \in u' }$$
                But observe that $\bvalue{ u = u' } = 1$ implies $\bvalue{ x \in u } = \bvalue{ x \in u' }$, so:
                \begin{align*}
                    \bvalue{ x \in u } &= \bvalue{ x \in u } \wedge \bvalue{ x \in u' } \\
                     &\leq \bvalue{ x \in u } \wedge \left( \bigvee\limits_{z' \in X_{u'}} \bvalue{ x = z'} \wedge \bvalue{ z' \in u' } \right) \\
                     &= \bigvee\limits_{z' \in X_{u'}} \bvalue{ x \in u } \wedge \bvalue{ x = z'} \wedge \bvalue{ z' \in u' }
                \end{align*}
                That is, $\delta_u(x,x) \leq \bigvee\limits_{z' \in X_{u'}} \lambda_{u,u'}(x,z')$.
            \end{itemize}
        \end{enumerate}
        
        Finally, we verify that $\lambda_{u,u'}$ is an isomorphism, with inverse morphism $\lambda_{u,u'}^{-1} = \lambda_{u',u}: \tuple{X_{u'},\delta_{u'}} \to \tuple{X_u, \delta_u}$.
        For all $x,y \in X_u$, 
        \begin{multline*}
            (\lambda_{u',u} \circ \lambda_{u,u'})(x,y) = \bigvee\limits_{x' \in X_{u'}} \lambda_{u,u'}(x,x') \wedge \lambda_{u',u}(x',y) = \\
            = \bigvee\limits_{x' \in X_{u'}} \bvalue{ x \in u } \wedge \bvalue{ x = x' } \wedge \bvalue{ x' \in u' } \wedge \bvalue{ y \in u } \wedge \bvalue{ y = x' } \wedge \bvalue{ x' \in u' } \leq \\
            \leq \bvalue{ x \in u } \wedge \bvalue{ x = y } \wedge \bvalue{ y \in u } = \delta_u(x,y)
        \end{multline*}
        Therefore, using Proposition \ref{prop:igualdade-morf-hsets}, we conclude that $\lambda_{u',u} \circ \lambda_{u,u'} = id_{\tuple{X,\delta}}$. Analogously, it can be verified that $\lambda_{u,u'} \circ \lambda_{u',u} = id_{\tuple{X',\delta'}}$.
        
        Now, for each $[u] \in \Seth{\H}$, let $I^{[u]}$ be the category given by:
        $$\Obj{I^{[u]}} \defeq [u]_m \qquad \qquad \Arr{I^{[u]}} \defeq [u]_m \times [u]_m$$
        where $[u]_m$ is the equivalence class of the elements with minimum rank. Consider the functor $F^{[u]}: I^{[u]} \to \hset$ such that
        $$F^{[u]}(u') \defeq \tuple{X_u,\delta_u}, \text{ for all } u' \in [u]_m$$
        $$F^{[u]}(u',u'') \defeq \lambda_{u',u''}: \tuple{X_{u'}, \delta_{u'}} \to \tuple{X_{u''}, \delta_{u''}}, \text{ for all } u', u'' \in [u]_m$$
        At last, we may define the functor $\Psi: \Seth{\H} \to \hset$ as $\Psi([u]) = \lim\limits_{u' \in [u]_m} F^{[u]}(u')$.
        
        This functor can also be described more explicitly. The product of a family of $\H$-sets $\set{\tuple{X_i, \delta_i} \mid i \in I}$ is given by $\tuple{P, \delta}$, where the set is simply the Cartesian product $P = \prod\limits_{i \in I} X_i$ and $\delta: P \times P \to \H$ is given by:
        $$\delta\left( \tuple{x_i}_{i \in I}, \tuple{x_i'}_{i \in I} \right) = \bigwedge\limits_{i \in I} \delta(x,x')$$
        The projections $\pi_j: P \times X_j \to \H$ are given by
        $$\pi_j \left(\tuple{x_i}_{i \in I}, x'_j \right) = \delta_j(x_j,x'_j)$$
        for each $j \in I$ (see \cite{Bor08c}, exercise 2.13.15). The equalizer of two morphisms $\phi, \psi: \tuple{X,\delta} \to \tuple{X', \delta'}$ of $\H$-sets is $\tuple{X,\tau}$, where
        $$\tau(x,y) = \bigvee\limits_{x' \in X'} \phi(x,x') \wedge \psi(y,x')$$
        (see \cite{Bor08c}, exercise 2.13.16).
        
        We can then use the construction of limits by products and equalizers (see \cite{Bor08a}, Theorem 2.8.1), denoting $\Psi([u])$ by $\lim F^{[u]}$:
        \begin{diagram}
             & & & \tuple{X_{u''}, \delta_{u''}} & \\
             & &\ruTo(1,2)^{\pi'_{u''}}  & & \luTo(1,2)^{\pi''_{(u',u'')}} \\
            \lim F^{[u]} & \rTo^{\quad E \quad} & \prod\limits_{u' \in [u]_m} \tuple{X_{u'}, \delta_{u'}} & \pile{\rTo^\alpha \\ \rTo_\beta} & \prod\limits_{(u',u'') \in I^{[u]}_1} \tuple{X_{u''}, \delta_{u''}} \\
             & & \dTo<{\pi'_{u'}} & & \dTo>{\pi''_{(u',u'')}} \\
             & & \tuple{X_{u'}, \delta_{u'}} & \rTo^{\qquad \lambda_{u',u''} \qquad} & \tuple{X_{u''}, \delta_{u''}}
        \end{diagram}
        where $\pi', \pi''$ are the projections (of the corresponding products) and $\tuple{\lim F, E}$ is the equalizer of $\alpha$ and $\beta$, which are the morphisms that make the diagram commute. That is, 
        $$\pi''_{(u',u'')} \circ \alpha = \pi'_{u''} \qquad \qquad \qquad \pi''_{(u',u'')} \circ \beta = \lambda_{u',u''} \circ \pi'_{u'}$$
        
        We can proceed similarly for the arrows of the category. For each $f \in \V{\H}$ such that $\bvalue{\fun{f: u \to v}} = 1$, define $\lambda_f: \tuple{X_u, \delta_u} \to \tuple{X_v, \delta_v}$ as:
        $$\lambda_f(x,y) = \bvalue{ x \in u } \wedge \bvalue{ (x,y) \in f } \wedge \bvalue{ y \in v }$$
        Now, given $f' \in \V{\H}$ such that $\bvalue{\fun{f': u' \to v'}} = 1$ and $\bvalue{f = f'} = 1$ (which already implies $u \equiv u'$ and $v \equiv v'$), we obtain the following commutative diagram:
        \begin{diagram}
            \tuple{X_u, \delta_u} & \rIso^{\quad \lambda_{u,u'} \quad} & \tuple{X_{u'}, \delta_{u'}} \\
            \dTo<{\lambda_f} & & \dTo>{\lambda_{f'}} \\
            \tuple{X_v, \delta_v} & \rIso^{\lambda_{v,v'}} & \tuple{X_{v'}, \delta_{v'}}
        \end{diagram}
        Thus, we may define an arrow $\Psi([f]): \lim\limits_{u' \in [u]_m} F(u') \to \lim\limits_{v' \in [v]_m} F(v')$.
        
        \begin{theorem}\label{teo:equiv-hset-quo}
            The functors $\Phi,\Psi$ constructed above define an equivalence of categories: $\hset \simeq \Seth{\H}$.
        \end{theorem}

\section{Induced morphisms in Heyting valued models}\label{sec:induced}



    Previously, we saw (see Definition \ref{immer-def}) an injection $V\to\V\B$ given by $\hat\cdot$ which preserves the truth values of $\Sigma_1$ formulas (see Corollary \ref{immer-cor}). Currently, it is known that if $\phi: \A\to\B$ is a complete and injective morphism of Heyting algebras, we can define a map $\tilde\phi:\V\A\to\V\B$ that is injective and such that: for all $x,y\in\V\A$,
    
	$$ \phi\bvalue{x =  y}_\A=\bvalue{\tilde\phi(x) = \tilde\phi(y)}_\B$$
	$$ \phi\bvalue{x\in y}_\A=\bvalue{\tilde\phi(x)\in\tilde\phi(y)}_\B$$

    For $\Delta_0$ formulas, the equality, trivially, still holds. One gets the following inequality for any $\Sigma_1$ formula $\psi$:

	$$ \phi\bvalue{\psi(x_1,\cdots, x_n)}_\A\leq\bvalue{\psi(\tilde\phi(x_1),\cdots,\tilde\phi(x_n))}_\B$$
    
    It is relatively straightforward to relax these conditions to injective functions that preserve only arbitrary suprema and finite infima\footnote{
		Geometric morphisms or Locale morphisms, which are related to Topoi and Sheaves over Locales.
	} to obtain the useful inequalities:

	$$ \phi\bvalue{x =  y}_\A\leq\bvalue{\tilde\phi(x) = \tilde\phi(y)}_\B$$
	$$ \phi\bvalue{x\in y}_\A\leq\bvalue{\tilde\phi(x)\in\tilde\phi(y)}_\B$$

	And we still have the inequality for $\Sigma_1$ formulas.

    Our efforts were in providing a possible generalization of this construction for non-injective maps that preserve arbitrary suprema and finite infima. In this section, we focus on that and on some difficulties we faced in the process.
    
    The reason for our search is that it is taken as a fact that the category of Heyting/Boolean valued models is related to other categories endowed with these morphisms. Despite us having a \emph{horizontal} connection between models and topoi
    $$\H \rightsquigarrow \V{\H} \rightsquigarrow \mathbf{Set}^{(\H)} \simeq \hset$$
    the vertical connections between arrows from $\H\to\H'$, $\V\H\to\V{\H'}$, etc. does not seem to have been widely explored in the literature. The only studied cases were automorphisms of complete Boolean algebra, complete monomorphisms between complete Boolean algebras (see exercise 3.12 in \cite{Bel05}) and retractions associated to those morphisms (see chapter 3 of \cite{Gui13}).
    
    This constitutes one of our main motivations to study if (and how) we could induce arrows between models from more general arrows between complete Heyting algebras. The other one is purely categorical: can geometric morphisms between localic topoi be lifted to morphisms between their associated Heyting value models?
    
    
    Here we present the main results of our work: we consider and explore how more general kinds of morphisms between complete Heyting algebras $\HH$ and $\HH'$ induce arrows between $V^\HH$ and $V^{\HH'}$, and   between their corresponding Heyting topoi $\Seth{\HH} (\simeq \Sh{\HH})$ and $\Seth{\HH'} (\simeq \Sh{\HH'})$. 

In the remainder of the section, $\H$ and $\H'$ will denote complete Heyting algebras, and $f:\H\to\H'$ shall be a locale morphism (notation: $f\in\mathbf{Loc}(\H,\H')$), \emph{i.e.}, $f$ is a function that preserves arbitrary suprema and finite infima.

\subsection{Induced morphisms}

\begin{definition}\label{def:tildef}
\textbf{(First proposal)} We recursively define a family
$$\left\{\tilde{f}_\alpha: \V{\H}_\alpha \rightharpoonup \V{\H'}_\alpha \ \middle| \ \alpha \in \ord \right\}$$
where $\rightharpoonup$ indicates that $\tilde{f}_\alpha$ are ``semi-functions". That is, for all $x \in \V{\H}_\alpha$, there exists $x' \in \V{\H'}$ such that $\tuple{x,x'} \in \tilde{f}_\alpha$) in the following way: for every $\alpha \in \ord$ and every $\tuple{x,x'} \in \V{\H}_\alpha \times \V{\H'}$, $\tuple{x,x'} \in \tilde{f}_\alpha$ if, and only if, there exists a surjection $\varepsilon: \dom{x} \twoheadrightarrow \dom{x'}$ such that $\tuple{u,\varepsilon(u)} \in \tilde{f}_{\varrho(x)}$ for all $u \in \dom{x}$, and the following diagram commutes:
\begin{diagram}
    \dom{x} & \rTo^{\quad x \quad} & \H\\
    \dOnto<\varepsilon & & \dTo>f\\
    \dom{x'} & \rTo^{x'} & \H'
\end{diagram}
Under these conditions, we say that $\varepsilon$ witnesses $\tuple{x,x'} \in \tilde{f}_\alpha$. Note that, if we suppose (by induction) that $\tilde{f}_\beta$ is defined for every $\beta < \alpha$, then the semi-function $\tilde{f}_{\varrho(x)}: \V{\H}_{\varrho(x)} \rightharpoonup \V{\H'}_{\varrho(x)}$ is defined, therefore $\dom{x'} \subseteq \V{\H'}_{\varrho(x)}$ and $x' \in \V{\H'}_\alpha$.


Thus, we define $\tilde{f} \defeq \bigcup\limits_{\alpha \in \ord} \tilde{f}_\alpha$ and
$$dom\left(\tilde{f}\right) = \bigcup\limits_{\alpha \in \ord} dom\left(\tilde{f}_\alpha\right) = \bigcup\limits_{\alpha \in \ord} \V{\H}_\alpha = \V{\H}$$
so that $\tilde{f}$ is also a semi-function $\tilde{f}: \V{\H} \rightharpoonup \V{\H'}$.
\end{definition}

\begin{prop}\label{prop:tildef-injetora}
If $f$ is injective, then, for all $\alpha \in \ord$, $\tilde{f}_\alpha$ is an injective function.
\end{prop}

\begin{proof}
By induction. Suppose that $\tilde{f}_\beta$ is an injective function for all $\beta < \alpha$ and let $\tuple{x,x'} \in \tilde{f}_\alpha$. Then, $\varepsilon = \tilde{f}_{\varrho(x)}\restriction: \dom{x} \twoheadrightarrow \dom{x'}$ is a bijection, because it's surjective by definition, and, since $\tuple{u,\varepsilon(u)} \in \tilde{f}_{\varrho(x)}$ for all $u \in \dom{x}$, the induction hypothesis implies that $\varepsilon = \tilde{f}_{\varrho(x)}\restriction$ is injective. Therefore, using the commutative diagram from the definition, $x'$ is uniquely determined by $x' = f \circ x \circ \varepsilon^{-1}$, that is, $\tilde{f}_\alpha$ is a function. Besides, if $x \neq y$ (in $\V{\H}_\alpha$), then $f$ being injective implies
$$\tilde{f}_\alpha(x) = f \circ x \circ \varepsilon^{-1} \neq f \circ y \circ \varepsilon^{-1} = \tilde{f}_\alpha(y)$$
so that $\tilde{f}_\alpha$ is also injective.
\end{proof}

Hence, this function covers the result stated in \cite{Bel05}, exercise 3.12.

	\begin{remark}
	Naive attempts to extend this initial proposal are fated to fail, for, in the absence of injectivity, the defined relation is not a function.

	In fact, note that Definition \ref{def:tildef} presents a serious problem: 
it does not guarantee that $\dom{\tilde{f}}_\alpha = \V\H_\alpha$, for each $\alpha \in \ord$. For example, consider the Boolean algebras $\mathbf{2} = \{0,1\}$ and $\mathbf{4} = \{0, a, \neg a, 1\}$ (with $0 \neq a \neq 1$) and the function $f: \mathbf{4} \to \mathbf{2}$ given by $f(a) = 0$ and $f(\neg a) = 1$. Firstly,
\begin{center}
\begin{tabular}{c c}
    $\V{\mathbf{4}}_0 = \emptyset$ & $\V{\mathbf{2}}_0 = \emptyset$ \\
    $\V{\mathbf{4}}_1 = \{\emptyset\}$ & $\V{\mathbf{2}}_1 = \{\emptyset\}$ \\
    $\V{\mathbf{4}}_2 = \big\{ \{(\emptyset, 0)\}, \{(\emptyset, 1)\}, \{(\emptyset, a)\}, \{(\emptyset, \neg a)\} \big\}$ \hspace{2em} & \hspace{2em} $\V{\mathbf{2}}_2 = \big\{ \{(\emptyset, 0)\}, \{(\emptyset, 1)\} \big\}$
\end{tabular}
\end{center}
Let $x \defeq \big\{ \big(\set{(\emptyset, 0)}, 0\big) \big\}, \big(\set{(\emptyset, a)}, 1 \big) \big\}$, $u \defeq \{(\emptyset, 0)\}$ and $v \defeq \{(\emptyset, 1)\}$. It can be easily verified that $\{(\emptyset, 0)\}$ is the only element of $\V{\mathbf{2}}_2$ such that $\tuple{u, \{(\emptyset, 0)\}}, \tuple{v, \{(\emptyset, 0)\}} \in \tilde{f}_2$.

Thus, consider $x' \in \V{\mathbf{2}}$ and suppose there exists $\varepsilon: \dom{x} \twoheadrightarrow \dom{x'}$ such that $\tuple{u,\varepsilon(u)}, \tuple{v,\varepsilon(v)} \in \tilde{f}_2$, \emph{i.e.} $\varepsilon(u) = \varepsilon(v) = \{(\emptyset, 0)\}$. But in this case, we cannot guarantee the diagram in the definition commutes, since we would have:
$$0 = \varphi(x(u)) = x'(\varepsilon(u)) = x'(\varepsilon(v)) = \varphi(x(v)) = 1$$
Therefore, there does not exist $x' \in \V{\mathbf{2}}$ such that $\tuple{x,x'} \in \tilde{f}_3$.

\end{remark}

To deal with this issue, we add more elements to the image of the semi-function, closing it by the equivalence relation $\equiv$ (another option would be to close the images only for the equivalent members with minimum rank).

	\renewcommand{\epsilon}{\varepsilon}

	\begin{definition}{Generalized Connection between $\V\H$s}
	
		Let $f\in\mathbf{Loc}({\H},{\H'})$. Define the following compatible family of relations by ordinal recursion:
		
		\begin{align*}
			x \mathrel{\tilde f_\alpha} y \iff &\exists (\epsilon:\dom x\sur\dom y) : (y\circ\epsilon = f\circ x) \land\\
			&\forall u\in\dom x:\exists v\in\V{\H'}:\exists\beta<\alpha: (u\mathrel{\tilde f}_\beta v) \land \bvalue{v=\epsilon(u)} = 1 
		\end{align*}

		$$ \tilde f=\bigcup_{\alpha\in\ord} \tilde f_\alpha$$
    \end{definition}
    
    This definition in particular was used because of the following: the requirement of the existence of a surjective function is due to our need that every object that is related to $x$ has its (domain's) elements determined by elements of (the domain of) $x$. This is true in the injective case, where the function is the witness of this existential. In the non-injective case, $\epsilon$ ``essentially''\footnote{Up to $\bvalue{\cdot=\cdot}=1$-equivalence.}
    is going to be a ``piece'' or ``fragment'' of $\tilde f$ that happens to be a function and
    behaves \emph{similarly} to how $\tilde f$ would if $f$
    was injective.
    
    We demand that $y\circ\epsilon = f\circ x$, to extend the original idea of the construction by injective morphisms to more general ones:  $\tilde f(x)(\tilde f(u))=(f\circ x)(u)$, $u \in dom (x)$. If $y$ is related to $x$, then  there is a function fragment of $\tilde f$ which makes the above commute.

    It is, however, not enough to ask only this, since one such $y$ could be chosen \textit{ad hoc} without the members of its domain being related to the members of $x$'s. There is no hope for us to attain the imposed conditions of inequalities of atomic formulas, which depend recursively on the domains of the involved objects, if we do not impose some similarly recursive demands on the relation.

    Thus, the final condition says that for every member $u$ of $x$'s domain, there was some $u$ in some previous step which to which it was related. Surely $\tilde f$ is only very rarely a function, but 
    after taking the quotient by $\bvalue{\cdot=\cdot}$-equivalence it is a function.
	
	Were we to remove $\bvalue{v=\epsilon(u)}=1$, and simply require that $u\mathrel{\tilde f}_\beta \epsilon(u)$, the definition would coincide for injective functions, but in general the domain of $\tilde f$ as a relation would not be total, \emph{i.e.} it wouldn't be all of $\V{\H}$.

	\newcommand{\mor}{f}
	\newcommand{\ind}{\mathrel{\tilde{\mor}}}

	\begin{theorem}{$\tilde f$'s domain is total}

		This is: for all morphism $f:{\H}\to{\H'}$,
		$$ \forall x\in\V{\H} :\exists x'\in\V{\H'} : x\ind x' $$

		\begin{proof}
            The proof follows from the  2 facts below:

			\begin{fact}
			
				Suppose that $\forall u\in\dom x:\forall\kappa\in\ord:\exists(u':\kappa\inj {\V{\H'}}):\forall\alpha<\kappa:u\ind u'_\alpha$.
				In this case, it is trivial to see that there is an $X'\subset\V{\H'}$ such that there is a bijection $\epsilon:\dom x\to\dom {x'}=X'$ such that $\forall u\in\dom x: u\ind\epsilon (u)$.
				
				Thus, let $x' = f \circ x \circ \epsilon^{-1}$. It is evident that $x\ind x'$.
                Therefore, if there exists a proper class of elements to the right of every member of the domain of $x$, then there is some $x'$ such that $x\ind x'$.
			\end{fact}

			\begin{fact}
				Suppose that $x\not=\emptyset$ and that $\exists x'\in\V{\H'}: x\ind x'$.
				Let $u\in\dom x$, $u\in\dom{x'}$ such that $u\ind u'$ and consider
				$\epsilon:\dom x\sur\dom{x'}$ witnessing $x\ind x'$.

				Trivially, 
				$$\exists\alpha\in\ord:\forall\xi>\alpha:\forall t':= \left[u'\cup\set{\tuple{\hat\xi,0}}\right] \rar \bvalue{t'=u'}=1$$

                Simply because we are adding some object which was not in the domain of $u'$ whose value under $t$ will be $0$, and in the equality, the $0$ will be in the antecedent of the implication.

                So, for each ordinal bigger than $\alpha$, we obtain a different set which is equal to $u'$ with ``probability'' $1$, and thus, $x$ must have a proper class of elements $y$ such that $x\ind y$.
			\end{fact}

            Joining the previous results, we have:
            
			$$\forall x\in\V{\H}: x\not=\emptyset \rar ([\forall u\in\dom x: \exists u':u\ind u'] \rar \exists x':x\ind x' ) $$

			As the consequent is true when $x=\emptyset$ --- for $\emptyset\ind\emptyset$ --- we have:
			
			$$\forall x\in\V{\H}: [\forall u\in\dom x: \exists u':u\ind u'] \rar \exists x':x\ind x' $$

			By regularity:
			$$\forall x\in\V{\H}: \exists x':x\ind x' $$
		\end{proof}
	\end{theorem}







Alternatively, we could assume in the Definition \ref{def:tildef} that $\tuple{x,x'} \in \tilde{f}_\alpha$ if, and only if, there exists a witness $\varepsilon: \dom{x} \twoheadrightarrow \dom{x'}$ as in the original definition, and, for all $u \in \dom{x}$, there exists $u' \in \V{\H'}$ such that $\tuple{u,u'} \in \tilde{f}_{\varrho(x)}$ and $\bvalue{ \varepsilon(u) = u' }' = 1'$.

We choose the first condition mentioned for our final definition, however all are equivalent in the quotient by $\equiv$.

\noindent
\textbf{Definition \ref{def:tildef} (complement)} Adding to the original definition, we also assume that if $\tuple{x,x'} \in \tilde{f}_\alpha$ (that is, there exists a witness for that) and $\bvalue{ x' = y' }' = 1'$, then $\tuple{x,y'} \in \tilde{f}_\alpha$.

\begin{remark}
We observe that $\V{\H}$ is a proper class (for $\H \neq \{0\}$), since there exists an injection $V \rightarrowtail \V{\H}$. It can also be shown that, for all $x \in \V\H$, $\left\{y \in \V{\H} \mid \bvalue{ x = y } = 1\right\}$ is a proper class. Indeed, for all $\Sigma \subseteq \V\H$ such that $\Sigma \cap \dom{x} = \emptyset$, we may define $y_\Sigma: \dom{x} \cup \Sigma \to \H$ as:
\[y_\Sigma(u) = 
    \begin{cases}
    x(u) & \text{, if } u \in \dom{x} \\
    0 & \text{, if } u \in \Sigma
    \end{cases}
\]
so that $\bvalue{ x = y_\Sigma } = 1$.
\end{remark}

Finally, we show that $\tilde{f}_\alpha$ does actually have the desired domain; the argument is similar to the one used above to prove the injective case.

\begin{prop}
For all $\alpha \in \ord$, $\dom{\tilde{f}_\alpha} = \V{\H}_\alpha$ (and its image is closed by $\equiv$).
\end{prop}

\begin{proof}
By induction: suppose that, for all $\beta \in \ord$, with $\beta < \alpha$, $\dom{\tilde{f}_\beta} = \V{\H}_\beta$. Let $x \in \V{\H}_\alpha$, and notice that, for all $u \in \dom{x} \subseteq \V{\H}_{\varrho(x)}$, the image of $\tilde{f}_{\varrho(x)}(\{u\})$ is a proper class, since it is non-empty and closed by $\equiv$ (by definition). Therefore, using the axiom of replacement, we may define an injection $\varepsilon: \dom{x} \rightarrowtail \V{\H'}$ satisfying $\tuple{u,\varepsilon(u)} \in \tilde{f}_{\varrho(x)}$, which we may restrict to a bijection $\tau: \dom{x} \to \varepsilon(\dom{x})$. Hence, just take $x' \in \V{\H'}$ as $x': \varepsilon(\dom{x}) \to \H'$ given by $x' \defeq f \circ x \circ \tau^{-1}$, so that $\tuple{x,x'} \in \tilde{f}_\alpha$, and thus $\dom{\tilde{f}_\alpha} = \V{\H}_\alpha$.
\end{proof}

Note that with this new definition, for all $x \in \V\H_\alpha$ there exists $x' \in \V{\H'}$ such that $\tuple{x,x'} \in \tilde{f}_\alpha$ is witnessed by a bijection $\tau: \dom{x} \to \dom{x'}$ (not only a surjection). In fact, we could have assumed the existence of a bijective witness in the definition of $\tilde{f}_\alpha$, and again that would be equivalent to the other possible definitions in the quotient by ``$\bvalue{\cdot=\cdot}=1$'' equivalence relation.

\subsection{Semantical preservation results}

\begin{theorem}\label{teo:tildef-pres-atomicas}
For all $\tuple{x,x'},\tuple{y,y'},\tuple{z,z'} \in \tilde{f}$,
$$f\left(\bvalue{ y \in x }\right) \leq' \bvalue{ y' \in x' }' \qquad \text{ and } \qquad f\left(\bvalue{ x = z }\right) \leq' \bvalue{ x' = z' }'$$
\end{theorem}

\begin{proof}
The proof is by induction on the well-founded relation
$$\tuple{u,x} \prec \tuple{v,y} \iff (u = v \text{ and } x \in \dom{y}) \text{ or } (u \in \dom{v} \text{ and } x = y)$$
Let $\varepsilon: \dom{x} \twoheadrightarrow \dom{x'}$ satisfying the conditions of \ref{def:tildef}. Then:
\begin{align*}
    f\left(\bvalue{ y \in x }\right) &= f\left(\bigvee\limits_{u \in \dom{x}} x(u) \wedge \bvalue{ u = y } \right) & \text{(by definition)}\\
     &= \bigvee\limits_{u \in \dom{x}}^{\hspace{18pt} \prime} f(x(u)) \wedge' f\left(\bvalue{ u = y }\right) & \left(\text{since $f$ preserves $\wedge,\bigvee$}\right)\\
     &\leq \bigvee\limits_{u \in \dom{x}}^{\hspace{18pt} \prime} f(x(u)) \wedge' \bvalue{ \varepsilon(u) = y' } & \left(\text{induction hypothesis $\tuple{u,\varepsilon(u)} \in \tilde{f}$)}\right)\\
     &= \bigvee\limits_{u \in \dom{x}}^{\hspace{18pt} \prime} x'(\varepsilon(u)) \wedge' \bvalue{ \varepsilon(u) = y' } & \left(\text{using that $\tuple{x,x'},\tuple{u,\varepsilon(u)} \in \tilde{f}$}\right)\\
     &= \bigvee\limits_{u' \in \dom{x'}}^{\hspace{18pt} \prime} x'(u') \wedge' \bvalue{ u' = y' } & \text{(since $\varepsilon$ is surjective)}\\
     &= \bvalue{ y' \in x' }' & \text{(by definition)}
\end{align*}
Now, since $\H$ is a Heyting algebra, note that the fact that $f$ preserves meets implies that $f$ is increasing, and also implies that $f(a \to b) \leq f(a) \to f(b)$, for all $a,b \in \H$. With that, let $\tau: \dom{z} \twoheadrightarrow \dom{z'}$ satisfying the conditions of \ref{def:tildef}. Then:
\begin{align*}
    f&\left(\bvalue{ x = z }\right) = & \\
     &= f\left(\bigwedge\limits_{\substack{u \in \dom{x}\\v \in \dom{z}}} \left(x(u) \to \bvalue{ u \in z }\right) \wedge \left(z(v) \to \bvalue{ v \in x }\right) \right) & \text{(by definition)}\\
     &\leq \bigwedge\limits_{\substack{u \in \dom{x}\\v \in \dom{z}}}^{\hspace{18pt} \prime} f\left(x(u) \to \bvalue{ u \in z }\right) \wedge' f\left( z(v) \to \bvalue{ v \in x } \right) & \text{($f$ is increasing)}\\
     &\leq \bigwedge\limits_{\substack{u \in \dom{x}\\v \in \dom{z}}}^{\hspace{18pt} \prime} \left(f(x(u)) \to' f\left(\bvalue{ u \in z }\right)\right) \wedge' \left( f(z(v)) \to' f\left(\bvalue{ v \in x } \right) \right) & \text{(comments above)}\\
     &\leq \bigwedge\limits_{\substack{u \in \dom{x}\\v \in \dom{z}}}^{\hspace{18pt} \prime} \left(f(x(u)) \to' \bvalue{ \varepsilon(u) \in z' }' \right) \wedge' \left( f(z(v)) \to' \bvalue{ \tau(v) \in x' }' \right) & \text{(comments below)}\\
     &= \bigwedge\limits_{\substack{u \in \dom{x}\\v \in \dom{z}}}^{\hspace{18pt} \prime} \left(x'(\varepsilon(u)) \to' \bvalue{ \varepsilon(u) \in z' }' \right) \wedge' \left( z'(\tau(v)) \to' \bvalue{ \tau(v) \in x' }' \right) & \text{(elements of $\tilde{f}$)}\\
     &= \bigwedge\limits_{\substack{u' \in \dom{x'}\\v' \in \dom{z'}}}^{\hspace{18pt} \prime} \left(x'(u') \to' \bvalue{ u' \in z' }' \right) \wedge' \left( z'(v') \to' \bvalue{ v' \in x' }' \right) & \text{($\varepsilon,\tau$ are surjections)}\\
     &= \bvalue{ x' = z' }' & \text{(by definition)}
\end{align*}
In the fourth step, we use that the implication is increasing in the second coordinate, and that by the induction hypothesis we have:
$$f\left(\bvalue{ u \in z }\right) \leq \bvalue{ \varepsilon(u) \in z' }' \quad \text{ and } \quad f\left(\bvalue{ v \in x }\right) \leq \bvalue{ \tau(v) \in x' }'$$
\end{proof}

This result easily extends to positive formulas (with only $\wedge, \vee$) with bounded quantifiers (of the form $\exists u \in x$ and $\forall u \in x$), using that corollary 1.18 from \cite{Bel05} gives us that:
$$\bvalue{ \exists u \in x \ \varphi(u) } = \bigvee\limits_{u \in \dom{x}} x(u) \wedge \bvalue{ \varphi(u) }$$
$$\bvalue{ \forall u \in x \ \varphi(u) } = \bigwedge\limits_{u \in \dom{x}} x(u) \to \bvalue{ \varphi(u) }$$

\begin{corollary}
Let $\varphi$ be a positive formula with bounded quantifiers. Then, for all $\tuple{a_1,a_1'},...,\tuple{a_n,a_n'} \in \tilde{f}$, we have:
$$f\left(\bvalue{ \varphi(a_1,...,a_n) } \right) \leq' \bvalue{ \varphi(a_1',...,a_n') }'$$
\end{corollary}

\begin{proof}
By induction in the complexity of the formula. The initial case, for atomic sentences, was shown in the previous theorem, and the cases with $\wedge$ and $\vee$ are immediate from the fact that $f$ preserves finite meets and joins. For quantifiers, the proof is similar to last theorem's proof. Let $\tuple{x,x'} \in \tilde{f}$ with witness $\varepsilon: \dom{x} \twoheadrightarrow \dom{x'}$. Then:
\begin{align*}
    f \left(\bvalue{ \exists u \in x \ \varphi(u) } \right) &= f \left(\bigvee\limits_{u \in \dom{x}} x(u) \wedge \bvalue{ \varphi(u) }\right) & \text{(by definition)}\\
     &= \bigvee\limits_{u \in \dom{x}}^{\hspace{18pt} \prime} f(x(u)) \wedge' f\left(\bvalue{ \varphi(u) } \right) & \left(\text{since $f$ preserves $\wedge,\bigvee$}\right)\\
     &\leq \bigvee\limits_{u \in \dom{x}}^{\hspace{18pt} \prime} f(x(u)) \wedge' \bvalue{ \varphi(\varepsilon(u)) } & \left(\text{by hypothesis and $\tuple{u,\varepsilon(u)} \in \tilde{f}$)}\right)\\
     &= \bigvee\limits_{u \in \dom{x}}^{\hspace{18pt} \prime} x'(\varepsilon(u)) \wedge' \bvalue{ \varphi(\varepsilon(u)) } & \left(\text{using that $\tuple{x,x'},\tuple{u,\varepsilon(u)} \in \tilde{f}$}\right)\\
     &= \bigvee\limits_{u' \in \dom{x'}}^{\hspace{18pt} \prime} x'(u') \wedge' \bvalue{ \varphi(u') } & \text{(since $\varepsilon$ is surjective)}\\
     &= \bvalue{ \exists u' \in x' \ \varphi(u') }' & \text{(by definition)}
\end{align*}
Similarly, we have
\begin{align*}
    f \left(\bvalue{ \forall u \in x \ \varphi(u) } \right) &= f\left(\bigwedge\limits_{u \in \dom{x}} x(u) \to \bvalue{ \varphi(u) } \right) & \text{(by definition)}\\
     &\leq \bigwedge\limits_{u \in \dom{x}}^{\hspace{18pt} \prime} f\left(x(u) \to \bvalue{ \varphi(u) }\right) & \text{(since $f$ is increasing)}\\
     &\leq \bigwedge\limits_{u \in \dom{x}}^{\hspace{18pt} \prime} f(x(u)) \to' f\left(\bvalue{ \varphi(u) }\right) & \left(\text{since $f(a \to b) \leq f(a) \to f(b)$}\right)\\
     &\leq \bigwedge\limits_{u \in \dom{x}}^{\hspace{18pt} \prime} f(x(u)) \to' \bvalue{ \varphi(\varepsilon(u)) }' & \text{(induction hypothesis)}\\
     &= \bigwedge\limits_{u \in \dom{x}}^{\hspace{18pt} \prime} x'(\varepsilon(u)) \to' \bvalue{ \varphi(\varepsilon(u)) }' & \text{(using that $\tuple{x,x'} \in \tilde{f}$)}\\
     &= \bigwedge\limits_{u' \in \dom{x'}}^{\hspace{18pt} \prime} x'(u') \to' \bvalue{ \varphi(u') }' & \text{(since $\varepsilon$ is surjective)}\\
     &= \bvalue{ \forall u' \in x' \ \varphi(u') }' & \text{(by definition)}
\end{align*}
\end{proof}

\subsection{Functorial properties}

Another consequence of the previous theorem is that, if $\bvalue{ x = z } = 1_H$, then, since $1_H = \bigwedge \emptyset$, we obtain:
$$f\left(\bvalue{ x = z }\right) = f(1_H) = 1_{H'} \leq \bvalue{ x' = z' }'$$
that is, $\bvalue{ x' = z' }' = 1_{H'}$. Therefore, when we take the quotient by $\equiv$, the semi-function $\tilde{f}$ defines an object mapping $\overline{f}: \Seth{\H} \to \Seth{\H'}$.

\begin{prop}$ $
\begin{enumerate}
    \item $\overline{id_\H} = id_{\Seth{\H}}: \Seth{\H} \to \Seth{\H}$;
    \item if $f': \H' \to \H''$ preserves finite meets and arbitrary joins, then $\overline{f'} \circ \overline{f} = \overline{f' \circ f}: \Seth{\H} \to \Seth{\H''}$.
\end{enumerate}
\end{prop}

\begin{proof}$ $
\begin{enumerate}
    \item We show that, for all $\alpha \in \ord$, if $\tuple{x,y'} \in (\widetilde{id_\H})_\alpha$, then $\bvalue{ x = y' } = 1$. Suppose, inductively, that this is the case for all $\beta < \alpha$, and let $\varepsilon: \dom{x} \twoheadrightarrow \dom{x'}$ (with $\bvalue{ x' = y' } = 1$) witness $\tuple{x,y'} \in (\widetilde{id_\H})_\alpha$. Then, $x' \circ \varepsilon = id_\H \circ x = x$, and for all $u \in \dom{x}$, $\tuple{u,\varepsilon(x)} \in (\widetilde{id_\H})_{\varrho(x)}$, thereby $\bvalue{ \varepsilon(u) = u } = 1$ (using the induction hypothesis). Thus, for all $u \in \dom{x}$, since $\varepsilon$ is surjective we have:
    $$x(u) = x'(\varepsilon(u)) = x'(\varepsilon(u)) \wedge \bvalue{ \varepsilon(u) = u } \leq \bigvee\limits_{w' \in \dom{x'}} x'(w') \wedge \bvalue{ w' = u } = \bvalue{ u \in x' }$$
    Similarly, for all $v' \in \dom{x'}$, there exists $v \in \dom{x}$ such that $\varepsilon(v) = v'$, and 
    $$x'(v') = x'(\varepsilon(v)) = x(v) = x(v) \wedge \bvalue{ \varepsilon(v) = v } \leq \bigvee\limits_{w \in \dom{x}} x(w) \wedge \bvalue{ w = v' } = \bvalue{ v' \in x }$$
    Observe that $x(u) \leq \bvalue{ u \in x' }$ if, and only if, $1 \leq x(u) \to \bvalue{ u \in x' }$, for all $u \in \dom{x}$; that is, $\bigwedge\limits_{u \in \dom{x}} x(u) \to \bvalue{ u \in x' } = 1$. Analogously, $x'(v') \leq \bvalue{ v' \in x' }$ for all $v' \in \dom{x'}$ is equivalent to $\bigwedge\limits_{v' \in \dom{x'}} x'(v') \to \bvalue{ v' \in x } = 1$. Therefore:
    $$\bvalue{ x = x' } = \bigwedge\limits_{u \in \dom{x}} \left(x(u) \to \bvalue{ u \in x' } \right) \wedge \bigwedge\limits_{v' \in \dom{x'}} \left( x'(v') \to \bvalue{ v' \in x } \right) = 1$$
    Finally, since $\bvalue{ x = x' } = 1 = \bvalue{ x' = y' }$, we may conclude that $\bvalue{ x = y' } = 1$, as desired.
    
    Now, by the definition of $\equiv$, $\bvalue{ x = y' } = 1$ if, and only if, $[x] = [y']$. As a result, taking the quotient, $\tuple{[x],[y']} \in \overline{id_\H}$ if, and only if, $[x] = [y']$, hence $\overline{id_\H}$ is the identity in $\Seth{\H}$.
    \item Let  $\tuple{[x],[z'']} \in \overline{f'} \circ \overline{f}$, \emph{i.e.}, there exists $[y'] \in \Seth{\H}$ such that $\tuple{[x],[y']} \in \overline{f}$ and $\tuple{[y'],[z'']} \in \overline{f'}$. Consider $\varepsilon: \dom{x} \twoheadrightarrow \dom{x'}$, with $x' \in [y']$, a witness of $\tuple{[x],[y']} \in \overline{f}$, and $\varepsilon': \dom{x'} \twoheadrightarrow \dom{y''}$, with $y'' \in [z'']$, a witness of $\tuple{[y'],[z'']} = \tuple{[x'],[z'']} \in \overline{f'}$. Then, $\varepsilon' \circ \varepsilon: \dom{x} \to \dom{y''}$ witnesses $\tuple{[x],[z'']} \in \overline{f' \circ f}$. That is, we have shown that $\overline{f'} \circ \overline{f} \subseteq \overline{f' \circ f}$, and since both are functions, we obtain $\overline{f'} \circ \overline{f} = \overline{f' \circ f}$.
\end{enumerate}
\end{proof}


At last, using that $f\left(\bvalue{ y \in x }\right) \leq' \bvalue{ y' \in x' }'$, it can be shown that, if $\bvalue{\fun{h: x \to y}} = 1$, then $\bvalue{\fun{\tilde{f}(h): \tilde{f}(x) \to \tilde{f}(y)}} = 1'$, and:
$$\tilde{f}(id_x) = id_{\tilde{f}(x)}: \tilde{f}(x) \to \tilde{f}(x) \quad \text{ and } \quad \bvalue{ \fun{\tilde{f}(id_x): \tilde{f}(x) \to \tilde{f}(x)}}' = 1'$$
Besides, if $\bvalue{\fun{g: y \to z}} = 1$, then:
$$\tilde{f}(g \circ h) = \tilde{f}(g) \circ \tilde{f}(h): \tilde{f}(x) \to \tilde{f}(z) \quad \text{ and } \quad \bvalue{ \fun{\tilde{f}(g \circ h): \tilde{f}(x) \to \tilde{f}(z)}}' = 1'$$
That is, by taking the quotient, $\overline{f}: \Seth{\H} \to \Seth{\H'}$ actually defines a functor.

As we saw in the first section, a $\left(\wedge, \bigvee \right)$-preserving function between Heyting algebras induces a functor between the corresponding sheaf topos which preserves finite limits and arbitrary colimits (the left adjoint of a geometric morphism). More precisely, using the natural equivalences $\hset \simeq \Sh{\H}$ and $\H'\text{-}\Set \simeq \Sh{\H'}$, such a function $f: \H \to \H'$ gives rise to a function $\varphi_f: \hset \to \H'\text{-}\Set$ given by:
\begin{diagram}
    \tuple{X,\delta} & & \tuple{X,f \circ \delta}\\
    \dTo<\phi & \quad \longmapsto \quad & \dTo>{f \circ \phi}\\
    \tuple{Y,\tau} & & \tuple{Y, f \circ \tau}
\end{diagram}
where $f \circ \phi: X \times Y \to \H'$. Thus, we investigate how $\tilde{f}$ may induce a morphism of $\H'$-sets.

\begin{prop}
Let $\tuple{x,x'} \in \tilde{f}$ with $\varepsilon: \dom{x} \twoheadrightarrow \dom{x'}$ as witness. Consider the function $\varepsilon^{H'}: \dom{x} \times \dom{x'} \to \H'$ given by:
$$\varepsilon^{H'}(u,v') \defeq f\left(\bvalue{ u \in x } \right) \wedge' \bvalue{ \varepsilon(u) = v' }' \wedge' \bvalue{ v' \in x' }', \text{ for all } \tuple{u,v'} \in \dom{x} \times \dom{x'}$$
Then,  $\varepsilon^{H'}$ defines a morphism of $\H$-sets $\varepsilon^{H'}: (\dom{x},f \circ \delta_x) \to (\dom{x'},\delta_{x'})$ which does not depend on the choice of witness, where
$$\delta_x(u,v) \defeq \bvalue{ u \in x } \wedge \bvalue{ u = v }, \text{ for all } u,v \in \dom{x}$$
$$\delta_{x'}(u',v') \defeq \bvalue{ u' \in x' }' \wedge' \bvalue{ u' = v' }', \text{ for all } u',v' \in \dom{x'}$$
\end{prop}

Note that $\delta_x$ and $\delta_{x'}$ are exactly the ones used in the equivalence $\hset \simeq \Seth{\H}$ because, since $\bvalue{ u \in x } \wedge \bvalue{ u = v } \leq \bvalue{ v \in x }$, we have:
$$\bvalue{ u \in x } \wedge \bvalue{ u = v }  \wedge \bvalue{ v \in x } = \bvalue{ u \in x } \wedge \bvalue{ u = v }$$

\begin{proof}
We verify the four conditions that define a morphism of $\H$-sets. Let $u,v \in \dom{x}$ and $u',v' \in \dom{x'}$.
\begin{enumerate}[topsep=0pt]
    \item $\delta_{x'}(u',v') \wedge' \varepsilon^{H'}(u,v') \leq \varepsilon^{H'}(u,u')$. Indeed,
    \begin{align*}
        \delta_{x'}&(u',v') \wedge' \varepsilon^{H'}(u,v') =\\
         &= \bvalue{ u' \in x' }' \wedge' \bvalue{ u' = v' }' \wedge' f\left(\bvalue{ u \in x } \right) \wedge' \bvalue{ \varepsilon(u) = v' }' \wedge' \bvalue{ v' \in x' }'\\
         &\leq f\left(\bvalue{ u \in x } \right) \wedge' \bvalue{ \varepsilon(u) = u' }' \wedge' \bvalue{ u' \in x' }' \\
         &= \varepsilon^{H'}(u,u')
    \end{align*}
    \item $(f \circ \delta_x)(u,v) \wedge' \varepsilon^{H'}(u,v') \leq \varepsilon^{H'}(v,v')$. Indeed,
    \begin{align*}
        (f &\circ \delta_x)(u,v) \wedge' \varepsilon^{H'}(u,v') =\\
        &= f\left(\bvalue{ u \in x } \wedge \bvalue{ u = v }\right) \wedge' f\left(\bvalue{ u \in x } \right) \wedge' \bvalue{ \varepsilon(u) = v' }' \wedge' \bvalue{ v' \in x' }' \\
        &\leq f\left(\bvalue{ v \in x }\right) \wedge' \bvalue{ \varepsilon(u) = \varepsilon(v) }' \wedge' \bvalue{ \varepsilon(u) = v' }' \wedge' \bvalue{ v' \in x' }' \\
        &\leq f\left(\bvalue{ v \in x }\right) \wedge' \bvalue{ \varepsilon(v) \in v' }' \wedge' \bvalue{ v' \in x' }' \\
        &= \varepsilon^{H'}(v,v')
    \end{align*}
    \item $\varepsilon^{H'}(u,u') \wedge \varepsilon^{H'}(u,v') \leq \delta_{x'}(u',v')$. Indeed,
    \begin{align*}
        \varepsilon^{H'}&(u,u') \wedge \varepsilon^{H'}(u,v') = \\
        &= f\left(\bvalue{ u \in x } \right) \wedge' \bvalue{ \varepsilon(u) = u' }' \wedge' \bvalue{ u' \in x' }' \wedge' f\left(\bvalue{ u \in x } \right) \wedge' \\
        & \hspace{6.4cm} \wedge' \bvalue{ \varepsilon(u) = v' }' \wedge' \bvalue{ v' \in x' }' \\
        &\leq f\left(\bvalue{ u \in x } \right) \wedge' \bvalue{ u' = v' }' \wedge' \bvalue{ u' \in x' }' \wedge' \bvalue{ v' \in x' }' \\
        &\leq \bvalue{ v' \in x' }' \wedge' \bvalue{ u' = v' }' \\
        &= \delta_{x'}(u',v') 
    \end{align*}
    \item $\bigvee\limits_{w' \in \dom{x'}}^{\hspace{18pt} \prime} \varepsilon^{H'}(u,w') = (f \circ \delta_x)(u,u)$. Indeed,
    \begin{itemize}
        \item for all $w' \in \dom{x'}$, using that $f$ preserves $1$, we have:
        \begin{align*}
            \varepsilon^{H'}(u,w') &= f\left(\bvalue{ u \in x } \right) \wedge' \bvalue{ \varepsilon(u) = w' }' \wedge' \bvalue{ w' \in x' }' \\
             &\leq f\left(\bvalue{ u \in x } \right) \wedge' 1_{H'} = f\left(\bvalue{ u \in x } \right) \wedge' f(1_H) \\
             &= f\left(\bvalue{ u \in x } \right) \wedge' f\left(\bvalue{ u \in u } \right) =f\left(\bvalue{ u \in x } \wedge \bvalue{ u = u }\right) \\
             &= (f \circ \delta_x)(u,u)
        \end{align*}
         thus, $\bigvee\limits_{w' \in \dom{x'}}^{\hspace{18pt} \prime} \varepsilon^{H'}(u,w') \leq (f \circ \delta_x)(u,u)$;
         \item on the other hand, since $\varepsilon(u) \in \dom{x'}$,
         \begin{align*}
             \bigvee\limits_{w' \in \dom{x'}}^{\hspace{18pt} \prime} \varepsilon^{H'}(u,w') &\geq \varepsilon^{H'}\tuple{u,\varepsilon(u)} = \\
              &= f\left(\bvalue{ u \in x } \right) \wedge' \bvalue{ \varepsilon(u) = \varepsilon(u) }' \wedge' \bvalue{ \varepsilon(u) \in x' }' \\
              &= f\left(\bvalue{ u \in x } \right) \wedge' 1_{H'} \wedge' \bvalue{ \varepsilon(u) \in x' }' \\
              &= f\left(\bvalue{ u \in x } \right) \wedge' \bvalue{ \varepsilon(u) \in x' }' \\
              &\geq f\left(\bvalue{ u \in x } \right) \wedge' f\left(\bvalue{ u \in x } \right) = f\left(\bvalue{ u \in x } \right) \\
              &= f\left(\bvalue{ u \in x } \wedge 1_H\right) = f\left(\bvalue{ u \in x } \wedge \bvalue{ u = u }\right)\\
              &= (f \circ \delta_x)(u,u)
         \end{align*}
         Therefore, $\bigvee\limits_{w' \in \dom{x'}}^{\hspace{18pt} \prime} \varepsilon^{H'}(u,w') \geq (f \circ \delta_x)(u,u)$.
    \end{itemize}
\end{enumerate}
Finally, note that this result does not depend on the choice of witness: let $\tuple{u,v'} \in \dom{x} \times \dom{x'}$ and $\tau: \dom{x} \twoheadrightarrow \dom{x'}$ be a witness of $\tuple{x,x'} \in \tilde{f}$. Then, since $u \in \dom{x}$, we have $\tuple{u,\varepsilon(u)} \in \tilde{f}$ e $\tuple{u,\tau(u)} \in \tilde{f}$; and since $1 = \bvalue{ u = u }$, the previous theorem gives us:
$$1' = f(1) = f\left(\bvalue{ u = u }\right) \leq \bvalue{ \tau(u) = \varepsilon(u) }'$$
Thus,
\begin{align*}
    \varepsilon^{H'}(u,v') &= f\left(\bvalue{ u \in x } \right) \wedge' \bvalue{ \varepsilon(u) = v' }' \wedge' \bvalue{ v' \in x' }' = \\
    &= f\left(\bvalue{ u \in x } \right) \wedge' 1_{H'} \wedge' \bvalue{ \varepsilon(u) = v' }' \wedge' \bvalue{ v' \in x' }' \\
    &= f\left(\bvalue{ u \in x } \right) \wedge' \bvalue{ \tau(u) = \varepsilon(u) }' \wedge' \bvalue{ \varepsilon(u) = v' }' \wedge' \bvalue{ v' \in x' }' \\
    &\leq f\left(\bvalue{ u \in x } \right) \wedge' \bvalue{ \tau(u) = v' }' \wedge' \bvalue{ v' \in x' }' \\
    &= \tau^{H'}(u,v')
\end{align*}
thereby $\varepsilon^{H'}(u,v') \leq \tau^{H'}(u,v')$. The proof that $\tau^{H'}(u,v') \leq \varepsilon^{H'}(u,v')$ is analogous.
\end{proof}

\begin{remark}
The idea now would be to show that such morphisms $\varepsilon^{H'}$ are isomorphisms, which could be used to build a natural isomorphism between $\overline{f}$ and $\varphi_f$. To show that, a possibility would be to use the characterization of monomorphisms and epimorphisms in $\hset$ (see \cite{Bor08c}, Propositions 2.8.8 and 2.8.7), that is, to show that for all $u,v \in \dom{x}$ and $u' \in \dom{x'}$:
\begin{itemize}
    \item $\varepsilon^{H'}(u,u') \wedge \varepsilon^{H'}(v,u') \leq (f \circ \delta_x) (u,v)$ (which is equivalent to $\varepsilon^{H'}$ being monic);
    \item $\bigvee\limits_{w \in \dom{x}} \varepsilon^{H'} (w,u') = \delta_{x'} (u',u')$ (which is equivalent to $\varepsilon^{H'}$ being epic);
\end{itemize}
and, since $\hset$ is a topos, $\varepsilon^{H'}$ would be an isomorphism.

Now, expanding the definitions,
\begin{align*}
    \varepsilon^{H'}&(u,u') \wedge \varepsilon^{H'}(v,u') =\\
     &= f\left(\bvalue{ u \in x } \right) \wedge' \bvalue{ \varepsilon(u) = u' }' \wedge' \bvalue{ u' \in x' }' \wedge' f\left(\bvalue{ v \in x } \right) \wedge' \bvalue{ \varepsilon(v) = u' }' \wedge' \bvalue{ u' \in x' }' \\
     &\leq f\left(\bvalue{ u \in x } \right) \wedge' \bvalue{ \varepsilon(u) = u' }' \wedge' \bvalue{ \varepsilon(v) = u' }' \\
     &\leq f\left(\bvalue{ u \in x } \right) \wedge'\bvalue{ \varepsilon(u) = \varepsilon(v) }' \\
     & \qquad \text{ and we want to show that } \leq f\left(\bvalue{ u \in x } \wedge \bvalue{ u = v } \right) = (f \circ \delta_x) (u.v)
\end{align*}
\begin{align*}
    \bigvee\limits_{w \in \dom{x}} \varepsilon^{H'} (w,u') &= \bigvee\limits_{w \in \dom{x}} f\left(\bvalue{ w \in x } \right) \wedge' \bvalue{ \varepsilon(w) = u' }' \wedge' \bvalue{ u' \in x' }' \\
     &= \bigvee\limits_{w \in \dom{x}} f\left(\bvalue{ w \in x } \right) \wedge' \bvalue{ \varepsilon(w) = \varepsilon(t) }' \wedge' \bvalue{ \varepsilon(t) \in x' }' \text{($\varepsilon$ is surjective)} \\
     &\geq f\left(\bvalue{ t \in x } \right) \wedge' \bvalue{ \varepsilon(t) = \varepsilon(t) }' \wedge' \bvalue{ \varepsilon(t) \in x' }' \\
     & \text{ and we want to show that } =  \bvalue{ u' \in x' }' = \delta_{x'}(u',u')
\end{align*}
(the other inequality for the epimorphism condition is trivial, because of the meet's properties). 

These inequalities can be achieved whenever
$f$ preserves meets and preserves strictly the values of atomic formulas
that is: if $\tuple{x,x'}, \tuple{y,y'},\tuple{z,z'} \in \tilde{f}$, then
$$f\left( \bvalue{ x \in y } \right) = \bvalue{ x' \in y' }' \qquad \qquad f\left( \bvalue{ x = z } \right) = \bvalue{ x' = z'' }'$$

Therefore, observing the proof of the aforementioned theorem, note that we may obtain these inequalities (and, thus, that $\varepsilon^{\H'}$ is iso) at least in the case that $f: \H \to \H'$ preserves (strictly) the implication and both arbitrary meets and joins. With that hypothesis, we could also adapt the corollary to the theorem to obtain the strict preservation of $\H$-values of all formulas with bounded quantifiers.

\end{remark}

\section{Final Remarks and Future Works}\label{sec:final}

\begin{remark}
The categorical and semantical correspondences between local set theories (= topoi, see \cite{Bel88}) and cumulative (constructions in) set theories has been studied since the late 1970s : \cite{Fou80}, \cite{Hay81}, \cite{BS92}, \cite{Shu10}, \cite{Yam18}. It will be interesting to determine in what level  this semantical correspondence is compatible with the change of basis given by a locale morphism $f: \H \to \H'$. \\
Possible extensions of this correspondence to other kinds of categories associated to other complete lattices (eventually endowed with additional structure \cite{LT15}) could give us a clue of what are the ``right semantical notions" of the less structured side of the correspondence (\emph{i.e.}, the cumulative construction), since the notion of $\H$-set can be extended to more general algebras (\cite{Men19}).



\end{remark}

\begin{remark}
In a different direction, another aspect that could be analysed is if the ``lifting property" through $\V{\H} \twoheadrightarrow \Sh{\H}$ also holds for other natural topoi  morphisms, such as the {\em logical functors}. Since logical functors and (the left part of) geometric morphism   coincide only trivially (\emph{i.e.} iff when both are equivalences of categories), this will be in fact a new direction to pursue.

Note that the ``conceptual orthogonality" between the two kind of functors
$\Sh{\H} \to \Sh{\H'}$ occurs already in the algebraic level of arrows $\H \to \H'$. More precisely, given a non-trivial complete Boolean algebra $(\B, \leq)$ and the unique injective  morphism  $i : \mathbf{2} \hookrightarrow \B$ (where $\mathbf{2} = \{0,1\}$), we get three kinds of morphisms $\B \to \mathbf{2}$:
\begin{itemize}[topsep=0pt]
    \item $l: \B \to \mathbf{2}$ is the left adjoint of $i$ (given by  $l(x) = 0 \Leftrightarrow x = 0$): it preserves only the suprema;
    \item $r: \B \to \mathbf{2}$ is the right adjoint of $i$ (given by $r(x) = 1 \Leftrightarrow x = 1$): it preserves only the infima;
    \item $U: \B \to \mathbf{2}$ is the quotient by an ultrafilter $U$, that preserves $0$, $1$, negation, implication, finite sups and finite infs.
\end{itemize}

On the other hand, note that a logical functor  $\Sh{\H} \to \Sh{\H'}$ induces a Heyting algebra morphism $\H \to \H'$ (since $\H \cong Subobj(\mathbf{1})$). Therefore, we would expect to be able to establish a correspondence between other kind of morphisms $\H \to \H'$ and the logical functors $\Sh{\H} \to \Sh{\H'}$ and ask how they are related to some alternative notion of induced arrow $\V{\H} \to \V{\H'}$.

In particular, it seems be natural to consider the connections between the various ``forcing relations"  (according the previous remark),  classical and intuitionistic, related to the canonical morphisms between complete Boolean algebras associated to a complete Heyting algebra $Reg(\H) \hookrightarrow \H$  and $\H \twoheadrightarrow \displaystyle \frac{\H}{\langle x \leftrightarrow \neg\neg x \rangle}$.





\end{remark}

  

\newpage

\begin{thebibliography}{10}
    \bibitem[Bel05]{Bel05} John L. Bell. \textit{Set theory: Boolean-valued models and independence proofs}. 3rd ed. Oxfor Logic Guides, vol. 47. Oxford, United Kingdom: Clarendon Press, 2005.
    
    \bibitem[Bel88]{Bel88} John L. Bell. \textit{Toposes  and  local  set  theories:  an  introduction}. Oxford  Logic Guides, vol. 14. Oxford, United Kingdom: Clarendon Press, 1988.
    
    \bibitem[Bor08a]{Bor08a}  Francis  Borceux. \textit{Handbook  of  Categorical  Algebra}.  Vol.  1. Encyclopedia of mathematics and its applications, vol. 50. Cambridge, United Kingdom: Cambridge University Press, 2008.
    
    \bibitem[Bor08c]{Bor08c}  Francis  Borceux. \textit{Handbook  of  Categorical  Algebra}.  Vol.  3. Encyclopedia of mathematics and its applications, vol. 52. Cambridge, United Kingdom: Cambridge University Press, 2008.
    
    \bibitem[BS92]{BS92} Andreas R. Blass and Andre Scedrov. “Complete topoi representing models of set theory”. In: \textit{Annals  of Pure  and  Applied  Logic} 57 (1992), pp. 1–26. doi: https://doi.org/10.1016/0168-0072(92)90059-9.
    
    \bibitem[Fou80]{Fou80} Michael P. Fourman. “Sheaf models for set theory”. In: \textit{Annals  of Pure  and Applied Logic} 19 (1980), pp. 91–101. doi: https://doi.org/10.1016/0022-4049(80)90096-1.
    
    \bibitem[Gui13]{Gui13} Fiorella Guichardaz. “Limits of Boolean algebras and Boolean valued models”. MA thesis. Universitá degli studi di Torino, 2013.
    
    \bibitem[Hay81]{Hay81} Susumu  Hayashi.  ``On  set  theories  in  toposes''.  In: \textit{Logic  Symposia  Hakone 1979, 1980}. Ed. by Gert H. Müller, Gaisi Takeuti, and Tosiyuki Tugu\'e. Lecture Notes in Mathematics, vol. 891. Heidelberg, Germany: Springer-Verlag, 1981. Chap. 2, pp. 23–29.
    
    \bibitem[Jec03]{Jec03} Thomas  J.  Jech. \textit{Set  Theory.: The  third  millennium  edition,  revised  and  expanded}. Springer Monographs in Mathematics. Heidelberg, Germany: Springer-Verlag, 2003.
    
    \bibitem[Kun11]{Kun11} Kenneth Kunen. \textit{Set  Theory}. Studies in Logic, vol. 34. London, United Kingdom: College Publications, 2011.
    
    \bibitem[Lan98]{Lan98} Saunders Mac Lane. \textit{Categories for the Working Mathematician}. 2nd ed. Graduate texts in mathematics, vol. 5. New York, United States: Springer-Verlag, 1998.
    
    \bibitem[LT15]{LT15} Benedikt  L\"owe  and  Sourav  Tarafder.  “Generalized  algebra-valued  models  of set  theory”.  In: \textit{The  Review  of  Symbolic  Logic} 8  (2015),  pp.  192–205. doi: https://doi.org/10.1017/S175502031400046X.
    
    \bibitem[Men19]{Men19} Caio A. Mendes. “Sheaf-like categories and applications”. In private communication. 2019.
    
    \bibitem[Shu10]{Shu10} Michael  A.  Shulman.  “Stack  semantics  and  the  comparison  of  material  and structural set theories”. In: \textit{arXiv e-prints} (2010). arXiv: 1004.3802v1.
    
    \bibitem[Yam18]{Yam18} Keita  Yamamoto.  “Toposes  from  Forcing  for  Intuitionistic  ZF  with  Atoms”. In: \textit{arXiv e-prints} (2018). arXiv: 1702.03399v2.
\end{thebibliography}
\end{document}